\documentclass[sn-basic,Numbered]{sn-jnl}

\usepackage{graphicx}%
\usepackage{multirow}%
\usepackage{amsmath,amssymb,amsfonts}%
\usepackage{amsthm}%
\usepackage{mathrsfs}%
\usepackage[title]{appendix}%
\usepackage{xcolor}%
\usepackage{textcomp}%
\usepackage{manyfoot}%
\usepackage{booktabs}%
\usepackage{algpseudocode}%
\usepackage{listings}%
\usepackage[shortlabels]{enumitem}

\theoremstyle{thmstyleone}%
\newtheorem{theorem}{Theorem}
\theoremstyle{thmstyletwo}%
\newtheorem{remark}{Remark}%

\theoremstyle{thmstylethree}%

\usepackage[ruled]{algorithm2e}
\usepackage{bm}
\usepackage{bbm}
\usepackage{mathtools}
\usepackage{comment}
\usepackage{multirow}

\newtheorem{assumption}{Assumption}
\newtheorem{lemma}{Lemma}

\newcommand{\bbeta}{ \bm\beta}
\newcommand{\bu}{ \bm u}

\begin{document}
\title[Article Title]{Survival Analysis with Graph-Based Regularization for Predictors}

\author[1]{\fnm{Liyan} \sur{Xie}}\email{liyanxie@umn.edu}

\author[2]{\fnm{Xi} \sur{He}}\email{xi.isabel.he@gmail.com}

\author[2]{\fnm{Pinar} \sur{Keskinocak}}\email{pinar@isye.gatech.edu}

\author*[2]{\fnm{Yao} \sur{Xie}}\email{yao.xie@isye.gatech.edu}

\affil[1]{\orgdiv{Department of Industrial and Systems Engineering}, \orgname{University of Minnesota}, \orgaddress{ \country{United States.}}}

\affil[2]{\orgdiv{H. Milton Stewart School of Industrial and Systems Engineering}, \orgname{Georgia Institute of Technology}, \orgaddress{\country{United States}}}

\abstract{
We study the variable selection problem in survival analysis to identify the most important factors affecting survival time. Our method incorporates prior knowledge of mutual correlations among variables, represented through a graph. We utilize the Cox proportional hazard model with a graph-based regularizer for variable selection. We present a computationally efficient algorithm developed to solve the graph regularized maximum likelihood problem by establishing connections with the group lasso, and provide theoretical guarantees about the recovery error and asymptotic distribution of the proposed estimators. The improved performance of the proposed approach compared with existing methods are demonstrated in both synthetic and real organ transplantation datasets. }





\maketitle

\renewcommand{\thefootnote}{}
\footnotetext{\small \noindent This paper is dedicated to the memory of Professor Tze Leung Lai.}

\section{Introduction}
\label{sec:intro}

Survival analysis, a branch of statistics that deals with the analysis of time-to-event data, is a fundamental tool across various disciplines \cite{lai-book-survival,klein2003survival}. It has been used in various domains such as healthcare \citep{newcombe2017weibull,duan2018bayesian}, medical prognostic \citep{ohno2001modeling}, and manufacturing \citep{de2021data}. In such scenarios, the feature vector tends to be high-dimensional, and the variables tend to have complex correlations. 
This work studies variable selection in survival analysis with correlated covariates.
This is a fundamental problem in predictive modeling: when the feature vector is high-dimensional, it is crucial to select a subset of significant variables for model interpretation and predictability \citep{hastie2015statistical}.

Variable selection arises in a wide range of applications, including genetics \citep{tachmazidou2010bayesian}, healthcare \citep{newcombe2017weibull,duan2018bayesian}, and epidemiology \citep{greenland1989modeling,walter2009variable}. 
For example, in organ transplantation, we are interested in knowing which variable is useful in predicting the post-transplant survival time of the patient. In such cases, efficient identification of the key variables will be useful for better decision-making. However, the variables tend to be highly correlated, and their mutual correlation can be represented by an undirected {\it graph} known from prior knowledge or estimated from data. For instance, a real-data example of the correlation structures among the predicting variables for the organ transplant dataset is illustrated in Figure \ref{fig: deceased}. When such prior structural knowledge is available, incorporating it in variable selection may yield more precise results \citep{yu2016sparse}. Instead of selecting individual variables, the graph structure enables us to utilize the neighborhood information to estimate or select the variables jointly. 


Motivated by this, in this paper, we study variable selection for survival analysis when the variables are correlated through a graph structure. We consider the Cox proportional hazard model \citep{coxregression} with a {\it graph regularizer} to incorporate graph correlation structure between variables in the presence of complete observations and right-censored data.
%
Cox proportional hazard model \citep{coxregression,cox1975partial} has been widely studied in survival analysis literature. One line of existing work on variable selection focuses on the Bayesian procedures for censored survival data by applying different prior distributions on the coefficients. In \cite{faraggi1997large,faraggi1998bayesian}, the partial likelihood function is considered to avoid specifying the unknown baseline hazard function, and a normal prior is used to estimate the regression coefficients. In \cite{ibrahim1999bayesian}, the full likelihood function is considered, and they specify a nonparametric prior for the baseline hazard and a parametric prior for the regression coefficients. In \cite{giudici2003mixtures}, the mixtures of products of Dirichlet process priors are used to compare the explanatory power of each covariate.  In \cite{lee2011bayesian}, a special shrinkage prior based on normal and Gamma distributions is used to handle cases when the explanatory variables are of very high dimension. In \cite{nikooienejad2020bayesian}, a mixture prior is used, which is composed of a point mass at zero and an inverse moment prior. Other than using priors, a structure-based method using  Bayesian networks is proposed in \cite{lagani2010structure} for variable selection.

\begin{figure}[!t]
\begin{center}
\begin{tabular}{cc}
\includegraphics[width=0.3\textwidth]{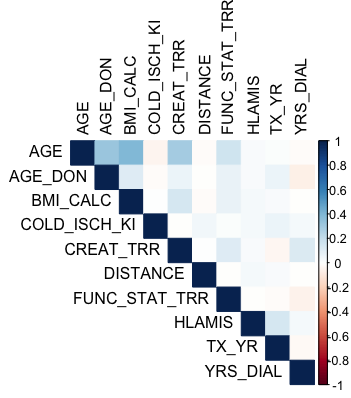} & \includegraphics[width=0.315\textwidth]{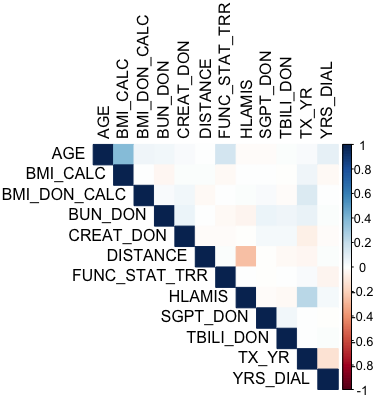}
\end{tabular}
\end{center}
\caption{Graph structure for correlation of variables in a pediatric kidney transplant data set: the inverse covariance matrix of the numerical variables in the living donor dataset (left) and the deceased donor dataset (right); more details are given in Section \ref{sec:data}.}
\label{fig: deceased}
\end{figure}


Another line of work performs variable selection through regularization, i.e., adding a regularization term into the likelihood function of the Cox model. 
The lasso regularization is applied for the Cox proportional hazard model in \cite{tibshirani1997lasso}. In \cite{fan2002variable}, the Smoothly Clipped Absolute Deviation (SCAD) regularization is proposed, and the resulted estimate is shown to have oracle property, i.e., the resulting estimate can correctly identify the true model; see \cite{fan2005overview} for an overview on such methods. In \cite{wu2012elastic}, the elastic net regularization is applied to the Cox model, and a solution path algorithm is developed; later in \cite{khan2016variable}, the adaptive elastic net is further applied for survival problems. In \cite{chaturvedi2014fused}, the fused lasso regularization is used, with applications to genomics. The adaptive lasso is applied in \cite{zhang2007adaptive}, and the consistency and convergence results are provided. In survival applications involving categorical variables, the group lasso regularization \citep{yuan2006model} is also often used; see \cite{kim2012analysis,utazirubanda2019variable} for examples.
In \cite{sun2014network}, a network-based regularizer is used for the Cox model, which contains an $\ell_1$ penalty and a quadratic Laplacian penalty. Compared with \cite{sun2014network}, which mainly penalizes the differences of coefficients between neighboring variables, we consider subgroups of the variable and impose an overall penalty for each group. It is also worth noting that non-parametric methods, such as the Kaplan-Meier estimator \citep{kaplan1958nonparametric}, Nelson-Aalen estimator \citep{nelson1972theory,aalen1978nonparametric}, and tree-based methods \citep{hothorn2004bagging, GRF-AoS}, provide flexible, model-free approaches for survival analysis. While these methods are effective in capturing non-linear effects and high-dimensional interactions, they are typically not designed for explicit variable selection.


The graph-based regularizer we use in this paper has been studied previously in \cite{yu2016sparse,li2020graph}, where only linear regression models are considered. 
We construct the variable graph through prior knowledge or estimating the correlation between predicting variables, represented as $G=(V, E)$ where $V$ and $E$ are the node and edge set, respectively. A node $v \in V$ represents a predicting variable and an edge $(u,v)$ connecting two nodes in the graph if the corresponding variables are correlated. 
We estimate the model parameter through the graph regularized maximum likelihood problem, which can be solved efficiently by connecting with group lasso (similar to the approach in \cite{yu2016sparse}). We establish the performance guarantee for model recovery error and the asymptotic normality of the estimated parameter. The good performance of the proposed method compared with baselines, such as methods without considering the graph correlation structure among variables, is demonstrated using both synthetic data and a real-data example of organ transplantation. 

We would like to remark that compared with linear regression problems with similar graph regularizer \citep{yu2016sparse}, the extension to survival analysis is non-trivial since the likelihood function under the Cox model is more involved than the least-square objective in linear regression. The log-likelihood function here can be represented by a counting process, and additional assumptions are needed to obtain desired local properties for the likelihood function. 
To address these challenges, we utilize techniques from survival analysis for Cox models \citep{fan2002variable}, and consider a general setting where covariates may change over time and rewrite the likelihood function using local asymptotic quadratic properties. 
In terms of the algorithm, the regularized likelihood can be optimized by adapting the predictor duplication method \citep{yu2016sparse,obozinski2011group} or efficiently using the fast iterative shrinkage thresholding algorithm \citep{beck2009fast}. To the best of our knowledge, our work is the first to apply such a graph regularizer for survival analysis, and it is a new contribution to variable selection in survival analysis. The introduction of a graph regularizer is crucial because, in many situations, including targeted healthcare applications, variables tend to have highly complex dependency structures. Our work provides an indispensable tool for performing variable selection in such cases.

The rest of the paper is organized as follows. 
Section~\ref{sec:lai} reviews the related seminal work by Professor Tze Leung Lai.  
Section~\ref{sec:prelim} reviews the preliminaries on survival analysis and graph-based regularization. 
Section~\ref{sec:method} presents the proposed regularization method for the Cox model, together with discussions on efficient algorithms for solving the coefficients estimate based on a predictor duplication method. 
Section~\ref{sec:theory} contains the main theoretical results, including guarantees for the accuracy and consistency of the maximum regularized likelihood estimate. 
Section~\ref{sec:simulation} contains numerical results comparing different methods using simulation. 
Section~\ref{sec:data} presents the application of the proposed method to two real data examples: the pediatric kidney transplant data and the primary biliary cirrhosis sequential data, and compares performance with other regularization methods. Section~\ref{sec:conclusion} concludes the paper. All proofs and additional numerical details are delegated to the appendix. 

\subsection{Lai's Work on Survival Analysis}\label{sec:lai}

Professor Tze Leung Lai has made substantial contributions to survival analysis among his diverse achievements in statistics. In this subsection, we provide a brief overview of his seminal work and its impact on our work. 
Professor Lai was instrumental in advancing the theory and methodology of survival analysis, mainly through his extensive studies on censored and truncated data. His research not only deepened the theoretical foundations of the field but also enhanced its applicability in various domains, including clinical trials and reliability engineering.

Professor Lai's research includes estimating the distribution function for a random variable $X$ based on censored or truncated observations. For censored data, observations take the form $(x_i \wedge t_i, \delta_i=\mathbbm{1}_{\{x_i \leq t_i\}})$ with $t_i$ being independent random variables and $\mathbbm{1}\{\cdot\}$ is the indicator function; for truncated data, observations are $(x_i, t_i)$ with variables observable only when $x_i\geq t_i$. In this paper, we mainly focus on right-censored data. Lai's work establishes statistical guarantees for the product-limit estimator of the distribution function of $X$ from such data  \cite{gu1990functional,lai1991estimating}, and for the non-parametric estimation of trimmed functionals of the conditional distribution of $X$ \cite{gross1996nonparametric}. Lai also extended the bootstrap method to the truncated and censored data \cite{gross1996bootstrap}.

Another significant area of Lai's research focuses on parameter estimation for censored linear regression models, where the response variable has a linear dependence on the covariate, modeled as $y_i = \bbeta^\top {\bf x}_i + \varepsilon_i$, and $y_i$ is subject to censoring or truncation \cite{lai1988stochastic}. The rank estimator for the slope $\bbeta$ was studied in 
\cite{lai1991rank,lai1992linear}, and the 
Modified Buckley-James Estimator was studied in \cite{lai1991large}. Furthermore, for censored or truncated regression with vector-valued coefficient $\bbeta$, various estimators have been proposed or studied, including the asymptotically efficient estimators \cite{lai1992asymptotically}, a bias-corrected least squares estimator \cite{lai1992asymptotic}, $M$-estimators \cite{lai1994missing,gross1996nonparametric,kim2000efficient}, and estimators resulted from estimating equations \cite{lai1995asymptotic}. 

Under the Cox proportional hazard model, which we utilize in this work, the confidence interval for the model parameter $\bbeta$ after a time-sequential test is established in \cite{lai2006confidence}. The confidence interval for the median survival times is also presented in \cite{lai2006confidence2}. An alternative approach that models the cumulative incidence function instead of the hazard function is proposed to obtain an asymptotically normal and efficient estimator of the regression parameter $\bbeta$ \cite{jin2017new}. Furthermore, Lai also studied the two-sample and sequential tests in clinical trials, comparing survival time between two treatment groups, with rank statistics based on censored data \cite{gu1991weak,gu1991rank,gu1998repeated}. The proposed Cox model with a graph-based regularizer can be potentially extended and applied to such test problems.

\section{Preliminaries}\label{sec:prelim}

\subsection{Cox Proportional Hazard Model}

We first introduce the basic notations in survival analysis and the Cox proportional hazard model \citep{coxregression} to be used. Denote $T$ as the event (failure) time. Throughout the text, organ transplantation can be used as an example to illustrate the models and methods, where $T$ refers to the post-transplant survival time of a patient. Assume $T$ is a random variable with cumulative distribution function (cdf) $F(t) = \mathbb P (T \leq t)$, and probability density function (pdf) $f(t) = F'(t) = dF(t)/dt$.
Define the event function as the upper tail probability
$S(t) = \mathbb P (T >t) = 1 - F(t)$,
and the hazard function as 
\begin{equation}\label{eq:hazard}
h(t) = \frac{f(t)}{S(t)} = - \frac{S'(t)}{S(t)} = -\frac{d (\log S(t))}{d t}.
\end{equation}
Denote the cumulative hazard function as $H(t) = \int_0^t h(u) du$. By taking integral on both sides of \eqref{eq:hazard}, we have
\[
S(t) = \exp \{-H(t)\}.
\]

Data is given in the form $(y_{1}, \delta_{1}, \bm{x}_{1}), \dots, (y_{n}, \delta_{n}, \bm{x}_{n})$, where $y_{i}$ is the time until the event, $\delta_{i} =1$ indicates a complete observation and $\delta_{i} =0$ a right-censored observation, $\bm{x}_{i} = [x_{i1},\dots,x_{ip}]^{\top}$ is the $p$-dimensional vector of predictors (covariates) for observation $i$, and $n$ denotes the sample size. For simplicity, assume that there are no tied event times. In the organ transplant example, an observation is a transplant (patient-organ pair),  $\bm{x}_{i}$ includes patient and donor/organ characteristics, and $y_i$ is the post-transplant survival time.
Given data $\{ (y_i,\delta_i,\bm{x}_i) \}_{i=1}^n$, by the definition of hazard in \eqref{eq:hazard}, the likelihood function is given by
\begin{equation}\label{eq:full_likelihood}
\begin{aligned}
L(\{ (y_i,\delta_i,\bm{x}_i) \}_{i=1}^n) & = \prod_{i: \delta_i = 1} f(y_i | \bm{x}_i) \prod_{i: \delta_i = 0} S(y_i | \bm{x}_i) = \prod_{i: \delta_i = 1} h(y_i | \bm{x}_i) \prod_{i=1}^n S(y_i | \bm{x}_i).
\end{aligned}
\end{equation}

Throughout this paper, we utilize the Cox proportional hazard model \citep{coxregression} defined as follows. The hazard function at time $t$, given covariate $\bm{x}_{i}$, takes the form
\begin{align}\label{eq:cox_hazard}
h(t|\bm{x}_{i}) = h_{0}(t)\exp(\boldsymbol{\beta_0}^{\top}\bm{x}_{i}),
\end{align}
where $h_{0}(\cdot)$ is the baseline hazard function, and $\boldsymbol{\beta_0} = [\beta_{0,1}, \dots, \beta_{0,p}]^{\top}$ is the vector of true coefficients. 
Let $H_0(t) = \int_0^t h_0(u) du$. Then the cumulative hazard function can be written as $H(t|\bm{x}_i) = H_0(t)\exp(\bm{\bbeta_0}^{\top}\bm{x}_i)$, and we have
\begin{equation}\label{eq:cox_survial_S}
   S(y_i | \bm{x}_i) =  \exp \{ -H_0(y_i)\exp(\bbeta_0^{\top}\bm{x}_i)\}. 
\end{equation}
Substitute \eqref{eq:cox_survial_S} and \eqref{eq:cox_hazard} into the likelihood function \eqref{eq:full_likelihood}, we obtain the full log-likelihood function under any parameter vector $\bm\beta$,
\begin{equation}\label{eq:full_llh}
\log L(\{ (y_i,\delta_i,\bm{x}_i) \}_{i=1}^n) = \sum_{i: \delta_i = 1} [\log h_0(y_i) + \bm{\bbeta}^{\top}\bm{x}_i] - \sum_{i=1}^n H_0(y_i)\exp(\bm{\bbeta}^{\top}\bm{x}_i).
\end{equation}
Our goal is to infer the unknown parameters $\bbeta$ from observations $\{ (y_i,\delta_i,\bm{x}_i) \}_{i=1}^n$. Taking the organ transplant as an example, we aim to identify which characteristics of the patient and donor/organ affect the post-transplant survival time the most.

\subsection{Partial Likelihood Function}

The baseline hazard function $h_0(\cdot)$ is usually {\it unknown} and has not been parameterized. Therefore, we adopt the commonly used partial likelihood function \citep{cox1975partial} instead of the full log-likelihood shown in \eqref{eq:full_llh}. The partial likelihood function is defined as the probability of the event being observed for observation $i$ at time $y_{i}$,
\begin{align*}
L_{i}(\boldsymbol{\beta}) = \frac{h(y_{i}|\bm{x}_{i})}{\sum_{j: y_{j}\ge y_{i}}h(y_{i}|\bm{x}_{j})} = \frac{\exp(\boldsymbol{\beta}^{\top}\bm{x}_{i})}{\sum_{j: y_{j}\ge y_{i}}\exp(\boldsymbol{\beta}^{\top}\bm{x}_{j})}.
\end{align*} 
Assuming independence of the observations, the joint partial likelihood function becomes
\begin{align*}
L(\boldsymbol{\beta}) = \prod_{i: \delta_{i} = 1}L_{i}(\boldsymbol{\beta})= \prod_{i: \delta_{i} = 1}\frac{\exp(\boldsymbol{\beta}^{\top}\bm{x}_{i})}{\sum_{j: y_{j}\ge y_{i}}\exp(\boldsymbol{\beta}^{\top}\bm{x}_{j})},
\end{align*} 
and the partial log-likelihood is given by
\begin{equation}\label{eq:likelihood}
\ell(\boldsymbol{\beta}) = \sum_{i=1}^{n} \delta_{i}\Big\{\boldsymbol{\beta}^{\top}\bm{x}_{i} -\log\Big(\sum_{j: y_{j}\ge y_{i}}\exp(\boldsymbol{\beta}^{\top}\bm{x}_{j})\Big) \Big\}.
\end{equation} 
Another interpretation of \eqref{eq:likelihood}, as given in \cite{fan2002variable}, is to substitute the ``least informative'' nonparametric prior for the unknown baseline cumulative hazard $H_0(\cdot)$. We use the formulation in \eqref{eq:likelihood} for optimization in the remainder of this paper. 


\subsection{Regularized Maximum Likelihood Estimation}\label{sec:regularizers}

In practice, the regularization-based method is commonly used to find the maximum likelihood fit of the Cox model. We solve the following optimization problem:
\begin{align}
\label{eqn:opt}
\min_{\boldsymbol{\beta}} -\frac{1}{n}\ell(\boldsymbol{\beta}) + g(\boldsymbol{\beta}),
\end{align}
where $\ell(\boldsymbol{\beta})$ is the partial log-likelihood function and $g(\boldsymbol{\beta})$ is the regularization term. A selected subset of potential representations for the regularization function $g(\boldsymbol{\beta})$ includes:
\begin{enumerate} 
\item Lasso \citep{tibshirani1997lasso}: $g(\boldsymbol{\beta}) = \lambda\Vert\boldsymbol{\beta}\Vert_{1}$, which encourages sparse solutions and the sparsity level can be controlled by the regularization parameter $\lambda>0$.
\item SCAD regularization \citep{fan2002variable}: $g(\boldsymbol{\beta}) = \sum_{j=1}^p f_\lambda(|\beta_j|)$, where \[f_\lambda'(\theta) = \mathbbm{1}(\theta \leq \lambda) + \frac{(a\lambda-\theta)_{+}}{(a-1)\lambda}\mathbbm{1}(\theta > \lambda), a >2, \ \theta > 0,\]
and $(x)_+ = \max\{x, 0\}$.
\item
Elastic net \citep{wu2012elastic}: $g(\bbeta) = \frac{\gamma}2\sum_{j=1}^p \beta_j^2 + \lambda\sum_{j=1}^p |\beta_j|$.

\item
Fused lasso \citep{chaturvedi2014fused}: $g(\bbeta) = \lambda_1 \sum_j |\beta_j| + \lambda_2\sum_{j=1}^{p-1}| \beta_{j+1} - \beta_j |$.

\item
Adaptive lasso \citep{zhang2007adaptive}: $g(\bbeta) = \lambda \sum_{j=1}^p \tau_j|\beta_j| $ with positive weights $\tau_j$.  

\item
Group lasso \citep{yuan2006model}: $g(\bbeta) = \lambda \sum_{k = 1}^{p}\Vert \bbeta_{\mathcal{I}_{k}}\Vert_{2}$, where $\mathcal{I}_{k}$ is the set of variables belonging to the $k^{\text{th}}$ group.
\end{enumerate}

It is worth mentioning that those classical regularization terms typically do {\it not} consider potential correlations between different predictors. When prior information about the relation between different predictor variables is available, represented by a predictor graph, we propose to adopt the graph-based regularization to estimate the model parameters better, as detailed in the following section.

\section{Graph-Based Regularization for Cox Model}\label{sec:method}

We first introduce the graph-based regularization given {\it known} or {\it pre-estimated} predictor graph for the {\it correlated} covariates in the Cox model. Then, we show that the regularized optimization problem can be solved efficiently.

Let $X = (\bm{x}_{1},\dots, \bm{x}_{n})^{\top} = (X_{1}, \dots, X_{p}) \in \mathbb{R}^{n\times p}$, with $X_{1}, \dots, X_{p}$ being column vectors and each column corresponds to a variable and its values across $n$ observations (i.e., patient-organ pairs). Assume a {\it known} covariance structure among $X_{1}, \dots, X_{p}$. 
For instance, in the organ transportation data set detailed in Section~\ref{sec:pediatric}, Figure~\ref{fig: deceased} shows an example of the correlation between predicting variables. To represent such correlations among the predictors, we can construct an {\it undirected} and unweighted graph $G=(V,E)$ where $V$ and $E$ denote the nodes and edges, respectively.
Such a graph can be constructed either by sample estimate or by domain knowledge.
There is a node $i\in V$ for each variable $i$ and an edge $(i,j)\in E$ if variables $i$ and $j$ are correlated. Let $E_G$ be the  matrix representing the edge set, where $E_G(i,j) = 1$ if $(i,j) \in E$ or $i=j$, and $0$ otherwise. Let $\mathcal{N}_{i}=\{j: E_G(i,j) = 1\}$ denote the neighbors of node $i$ and let $d_i = |\mathcal N_i|$ denote the cardinality of the set $\mathcal N_i$. 

%
%
\begin{remark}[Motivation for Graph-based Regularization]\label{rem1} Our usage of graph-based regularization is inspired by its usage in the linear regression setting as studied in \cite{yu2016sparse}. Here, we present a specific justification under the Cox model. Assume a random design covariate $\bm{x}\sim N(0_{p\times 1},\Sigma)$ and observation $(\bm{x},y)$, which for simplicity we assume uncensored, then the survival function in \eqref{eq:cox_survial_S} implies that $Z:= H_0(y)\exp\{\bbeta_0^\top \bm{x}\}$ follows the Exponential distribution with mean $1$, conditioned on $\bm{x}$, where $\bbeta_0$ is the true coefficients vector. Therefore,
\[
\Sigma_{xy}:= \mathbb E[- \bm{x} \log H_0(y)] = \mathbb E\left[ \bm{x} \cdot \mathbb E [ \bm{x}^\top \bbeta_0 - \log Z | \bm{x}] \right] = \mathbb E[ \bm{x}\bm{x}^\top ] \bbeta_0 = \Sigma \bbeta_0,
\]
which yields $\bbeta_0 = \Omega \Sigma_{xy}$. Here we denote $\Omega=\{\omega_{ij}\}_{i,j=1,2,\ldots,p}=\Sigma^{-1}$. By definition, the inverse covariance matrix $\Omega$ measures partial correlations among predictors, and $\Sigma_{xy}=(c_1,\ldots,c_p)^\top$ is a constant vector represents the marginal correlations between covariates $\bm{x}$ and $\log H_0(y)$, a function of the corresponding survival time. Consequently, we can decompose $\bbeta_0 = \Omega \Sigma_{xy}$ into $p$ parts: 
\[
\bbeta_0 = \sum_{i=1}^p c_i \Omega_{\cdot i},\ \text{where } \Omega_{\cdot i} \text{ is the $i$-th column of $\Omega$.}
\]
For a given predictor graph, we have $\omega_{ij}\neq 0$ if and only if the $i$-th and $j$-th predictor variable is uncorrelated. Therefore, the support of $\Omega_{\cdot i}$ is $\mathcal N_i$, the neighbors of node $i$. This motivates us to decompose the parameter $\bbeta$ into $p$ latent parts, with the $i$-th part supported only on the neighborhood set $\mathcal N_i$, as detailed next.    
\end{remark}

Before introducing our main optimization model, we first recall that our goal is to estimate the parameter $\bm\beta$ by solving a regularized optimization problem as shown in \eqref{eqn:opt}. 
From the insights in Remark \ref{rem1}, we adopt the following norm of $\bbeta$ which was used in \cite{yu2016sparse} by incorporating the additional correlation information on $X$ (captured by the graph $G$), for a given collection of non-negative weights $\bm\tau:=\{\tau_1,\tau_2,\ldots,\tau_p\}$:
\begin{equation}\label{eq:norm}
\Vert \bbeta \Vert_{G,\bm\tau} := \min_{\sum_{k=1}^pV^{(k)} = \bbeta,\ \text{supp}(V^{(k)}) \subseteq \mathcal N_k} \sum_{k=1}^p \tau_k \Vert V^{(k)} \Vert_2.
\end{equation}
Intuitively, $\bbeta$ is decomposed into $p$ terms: $\bbeta = \sum_{k=1}^p V^{(k)}$. For the $k$-th predictor variable, the corresponding term $V^{(k)}$ characterizes the effect of the $k$-th predictor variable on the survival time; if $V^{(k)}$ is non-zero, then the support of $V^{(k)}$ implies a connection between the $k$-th predictor and its neighbors in the graph $G$, i.e., $V^{(k)}$ is only supported on $\mathcal N_k$. The parameter $\tau_{k}\geq0$ is the regularization parameter that controls the importance of the regularization term $\Vert V^{(k)} \Vert_2$ for the $k$-th predictor. 
%
It can be verified that $\Vert \bbeta \Vert_{G,\bm\tau}$ satisfies the triangle inequality and is indeed a valid norm \citep{obozinski2011group}. 

Let the regularization term in \eqref{eqn:opt} be $g(\bm\beta) = \lambda\Vert \bbeta \Vert_{G,\bm\tau}$ for a regularization parameter $\lambda\geq0$, then we estimate the parameter $\bbeta$ by solving
\begin{equation}\label{eq:main}
\min_{\bbeta \in \mathbb{R}^p}  
-\frac{1}{n} \ell(\boldsymbol{\beta}) + \lambda \Vert \bbeta \Vert_{G,\bm\tau},
\end{equation}
which, by the definition in \eqref{eq:norm}, is also equivalent to
\begin{equation}\label{eqn:opt_graph}
\begin{aligned}
\min_{\boldsymbol{\beta}, V^{(1)}, \dots, V^{(p)}} & 
-\frac{1}{n} \ell(\boldsymbol{\beta}) + \lambda\sum_{k=1}^{p}\tau_{k}\Vert{V}^{(k)}\Vert_{2},\\
 \text{s.t.} \quad &
 \sum_{k= 1}^{p} V^{(k)} = \boldsymbol{\beta}, \quad \text{supp}(V^{(k)}) \subset \mathcal{N}_{k}, \ \forall k. 
\end{aligned}
\end{equation}

It is worth mentioning that the regularization term $\Vert \bbeta \Vert_{G,\bm\tau}$ is very general since it will be reduced to adaptive Lasso when there is no edge in the graph $G$, to group lasso when the graph $G$ has several disconnected complete subgraphs, and reduced to ridge regression when the graph is a complete graph \citep{yu2016sparse}.


The optimization problem in  (\ref{eqn:opt_graph}) could be reformulated to an unconstrained convex problem such that it can be solved efficiently using existing solvers. This technique is developed based on the {\it predictor duplication} technique in \cite{obozinski2011group}. For the observation $\bm{x}_i$ and the $k$-th predictor, let $\bm{x}_{\mathcal{N}_{k}}^{i}$ be the $|\mathcal{N}_{k}|\times 1$ subvector of $\bm{x}_{i}$, with indices from $\mathcal{N}_{k}$. Similarly, let $V^{(k)}_{\mathcal{N}_{k}}$ be the $|\mathcal{N}_{k}|\times 1$ subvector of $V^{(k)}$. Recall that we have the constraint $\text{supp}(V^{(k)}) \subset \mathcal{N}_{k}$, thus the subvector  $V^{(k)}_{\mathcal{N}_{k}}$ contains all non-zero values of the vector $V^{(k)}$. 
Then 
\[\sum_{i = 1}^{n}{\boldsymbol{\beta}}^{\top}\bm{x}_{i} = \sum_{i=1}^{n}\sum_{k=1}^{p}{V^{(k)}_{\mathcal{N}_{k}}}^{\top}\bm{x}^{i}_{\mathcal{N}_{k}},\] and the partial log-likelihood function in \eqref{eq:likelihood} can be rewritten as 
\begin{equation}\label{eqn:opt_pdm} 
\begin{aligned}
\ell(\boldsymbol{\beta}) = \sum_{i=1}^{n} \delta_{i}\Big\{ & \sum_{k=1}^{p}{V^{(k)}_{\mathcal{N}_{k}}}^{\top}\bm{x}^{i}_{\mathcal{N}_{k}}- \log\Big(\sum_{j: y_{j}\ge y_{i}}\exp(\sum_{k=1}^{p}{V^{(k)}_{\mathcal{N}_{k}}}^{\top}\bm{x}^{j}_{\mathcal{N}_{k}})\Big) \Big\}.
\end{aligned}     
\end{equation}

Therefore, the optimization problem \eqref{eqn:opt_graph} reduces to the unconstrained optimization problem with new duplicated variables $\{V_{\mathcal{N}_1}^{(1)}, \dots, V_{\mathcal{N}_p}^{(p)} \}$:
\[
\begin{aligned}
\min_{V_{\mathcal{N}_1}^{(1)}, \dots, V_{\mathcal{N}_p}^{(p)}}  \quad &-\frac{1}{n}   \sum_{i=1}^{n} \delta_{i}\Big\{
 \sum_{k=1}^{p}{V^{(k)}_{\mathcal{N}_{k}}}^{\top}\bm{x}^{i}_{\mathcal{N}_{k}}-\log\Big(\sum_{j: y_{j}\ge y_{i}}\exp(\sum_{k=1}^{p}{V^{(k)}_{\mathcal{N}_{k}}}^{\top}\bm{x}^{j}_{\mathcal{N}_{k}})\Big) \Big\} \\
&  + \lambda\sum_{k=1}^{p}\tau_{k}\Vert V^{(k)}_{\mathcal{N}_{k}} \Vert_{2},
\end{aligned}
\]
which can be solved using existing solvers for the group lasso regularization, such as the \texttt{R} package \texttt{grpreg} \citep{breheny2016package}.
After obtaining the optimal solution to the above unconstrained problem, denoted as ${\widehat{V}^{(k)}_{\mathcal{N}_{k}}},k=1,\ldots,p$, we let ${\widehat{V}^{(k)}_{\mathcal{N}_{k}^{c}}} = 0$ and the optimal parameter is $\widehat{\bbeta} = \sum_{k=1}^{p} {\widehat{V}^{(k)}}$. Note that when the neighborhood of some nodes is exactly the same, the decomposition of $\bbeta$ may not be unique. However, different decompositions lead to the same estimate $\bbeta$.  
Although the predictor duplication method is simple to use and can be solved using existing solvers, the dimension of the variables after duplication can be high when $p$ is high and the graph $G$ is dense. In such cases, we provide an alternative method by applying the fast iterative shrinkage thresholding algorithm (FISTA) \cite{beck2009fast} in Appendix \ref{app}.

\section{Theoretical Guarantees}\label{sec:theory}

In this section, we provide the theoretical properties for the estimate $\widehat\bbeta$ solved from \eqref{eqn:opt_graph}. 
Denote $\bbeta_0 =[ \beta_{01},\ldots,\beta_{0p}]^{\top}$ as the true parameters which is unknown, $J_0 = \{i: \beta_{0i} \neq 0 \}$ is the index of non-zero parameters, $J_0^c = \{ i: \beta_{0i} = 0 \}$ is the index of zero parameters, and $s_0 = |J_0|$ denotes the number of non-zero parameters. 

We first introduce some useful notations and results for the parameter $\bbeta$.  For any given $\bm\beta$ and non-negative weights vector $\bm\tau$, we note that the norm $\Vert \bbeta \Vert_{G,\bm\tau}$ of $\bm\beta$ as defined in \eqref{eq:norm} is computed based on the optimal decompositions $\{V^{(1)},\ldots, V^{(p)}\}$ of $\bbeta$ such that  $\bbeta=\sum_{k=1}^pV^{(k)}$ and $ \text{supp}(V^{(k)}) \subseteq \mathcal N_k$ for each $k$. Let $\mathcal U(\bbeta)$ denotes the set of all such decompositions of $\bbeta$ that minimizes $\sum_{k=1}^p \tau_k \Vert V^{(k)} \Vert_2$. In other words, $\mathcal U(\bbeta)$ consists of all optimal solutions to the optimization problem \eqref{eq:norm}:
\[
\begin{aligned}
\mathcal U(\bbeta) = \Big\{   \{V^{(1)},\ldots, V^{(p)}\}:\ &  \bbeta=\sum_{k=1}^pV^{(k)},\text{supp}(V^{(k)}) \subseteq \mathcal N_k, \\
&\sum_{k=1}^p \tau_k \Vert V^{(k)} \Vert_2 =  \Vert \bbeta \Vert_{G,\bm\tau}\Big\}.
\end{aligned}
\]

Denote $T,C,\bm{x}$ as the survival time, censoring time, and the associated covariates, respectively. 
We adopt the usual {\it counting process} notation and theory to derive the theoretical guarantee for the resulting estimate. More specifically, define the counting process $N_i(t) = \mathbbm{1}\{T_i \leq t, T_i \leq C_i \}$ and the indicator for being at risk $Y_i(t)= \mathbbm{1}\{ T_i \geq t, C_i \geq t \}$ for $i=1,\ldots,n$. Without loss of generality, we only consider the bounded time horizon $[0,1]$. The results can be extended to unbounded time interval $[0,\infty)$ and the general setting where the covariate $\bm{x}(t)$ can vary over time \citep{andersen1982cox}. Then, the partial log-likelihood function in \eqref{eq:likelihood} can be rewritten by the counting process as
\[
\begin{aligned}
\ell(\boldsymbol{\beta}) = & \sum_{i=1}^{n} \int_0^1\boldsymbol{\beta}^{\top}\bm{x}_{i}dN_i(t) -\int_0^1 \log\Big(\sum_{i=1}^nY_i(t)\exp\left(\boldsymbol{\beta}^{\top}\bm{x}_{i}\right)\Big) d\bar N(t),
\end{aligned}
\]
where $d N_i(t)$ is the increment over the infinitesimal interval $[t,t+dt)$ and it is either zero or one for the counting process $N_i$; $\bar N = \sum_{i=1}^nN_i$, and $d \bar N(t) = \sum_{i=1}^n d N_i(t)$. 

To simplify the notation, we define 
\begin{equation}\label{eq:def_s}
\begin{aligned}
s^{(0)} (\bbeta,t) &= \mathbb E\left[Y(t)\exp\{\bbeta^{\top}\bm x\}\right],  \\
s^{(1)} (\bbeta,t) &= \mathbb E \left[ Y(t)\bm x \exp\{\bbeta^{\top}\bm x\} \right],  \\
s^{(2)} (\bbeta,t) &= \mathbb E \left[ Y(t) \bm x\bm x^{\top}\exp\{\bbeta^{\top}\bm x\} \right].
\end{aligned}    
\end{equation}
For the Cox model, we adopt similar assumptions for the partial log-likelihood function as in \cite{fan2002variable,sun2014network}. 
\begin{assumption}[Assumptions for the partial likelihood function]\label{assup1} 
We assume that:
\begin{enumerate} 
\item $\int_0^1 h_0(t) dt < \infty $, and $\mathbb P\{ Y(t) = 1, \forall t\in[0,1] \} > 0$.
\item Covariates $x_j$, $j=1,\ldots,p$ are bounded and there exists a constant $M>0$ such that $\left\Vert\bm{x}\right\Vert_1 \leq M$. 
\item 
 There exists a neighborhood $\mathcal B\subset \mathbb{R}^{p}$ of $\bbeta_0$ such that
\[
\mathbb E \left\{\sup_{t\in[0,1], \ \bbeta \in \mathcal B} Y(t)\bm x^{\top} \bm x \exp\{\bbeta^{\top}\bm x\}\right\} < \infty. 
\]
\item $s^{(0)} (\cdot,t)$, $s^{(1)} (\cdot,t)$, $s^{(2)} (\cdot,t)$ are continuous in $\bbeta \in \mathcal B$, uniformly in $t\in[0,1]$; $s^{(0)}, s^{(1)}, s^{(2)}$ are bounded on $\mathcal B \times [0,1]$; $s^{(1)}$ is bounded away from zero on $\mathcal B \times [0,1]$. 
The information matrix 
\[
\begin{aligned}
&I(\bbeta_0) \!=\! \int_0^1  \Bigg(\frac{s^{(2)} (\bbeta_0,t) }{s^{(0)} (\bbeta_0,t) } \!-\!  \big( \frac{s^{(1)} (\bbeta_0,t) }{s^{(0)} (\bbeta_0,t) }  \big)\big( \frac{s^{(1)} (\bbeta_0,t) }{s^{(0)} (\bbeta_0,t) }  \big)^{\top}\Bigg) s^{(0)} (\bbeta_0,t)h_0(t)dt
\end{aligned}
\]
is {\it positive definite}. 
\end{enumerate}
\end{assumption}

The reason for imposing the above assumptions is to obtain the local asymptotic quadratic property for the partial likelihood function $\ell(\bbeta)$, as well as the asymptotic normality of the maximum partial likelihood estimates \citep{andersen1982cox, murphy2000profile}. More specifically, Assumptions \ref{assup1} (1), (3), and (4) are standard for the asymptotic theory of Cox models and are identical to assumptions imposed in \cite{fan2002variable}. Assumption \ref{assup1} (2) is similar to the Condition (C2) in \cite{sun2014network}. Moreover, the boundedness condition in Assumption \ref{assup1} (3) can be easily satisfied for bounded covariates $\bm x$ as imposed by Assumption \ref{assup1} (2). 

We also make the following assumptions for the true predictor graph $G$, which represents the underlying correlated structure among all predictors.
\begin{assumption}[Assumptions for the predictor graph $G$]\label{assup2} We impose the following assumptions for the predictor graph $G$:
\begin{enumerate} 
\item The neighborhood $\mathcal N_k \subseteq J_0$, $\forall k \in J_0$.
\item  There exists a neighborhood $\mathcal B\subset \mathbb{R}^{p}$ of $\bbeta_0$ and $\kappa > 0$ such that
\[
\begin{aligned}
&\inf_{\substack{\bbeta\in\mathcal B\\ \bm\xi \in \mathbb R^p \backslash \{0\}\\|J|\leq s_0}} 
\inf_{\substack{(V^{(1)},V^{(2)},\ldots,V^{(p)}) \in \mathcal{U}(\bm\xi) \\ \sum\limits_{k\notin J}\tau_k \Vert V^{(k)} \Vert_2 \leq 3\sum\limits_{k\in J}\tau_k \Vert V^{(k)} \Vert_2 }}
\!\!\frac12 \frac{  (\sum_{k=1}^p V^{(k)})^{\top} I(\bbeta) (\sum_{k=1}^p V^{(k)}) }{(\sum_{k\in J} \tau_k \Vert V^{(k)} \Vert_2)^2} 
\geq \kappa.
\end{aligned}
\]
\end{enumerate}
\end{assumption}

The Assumption \ref{assup2} (1) assumes that the predicted graph $G$ is {\it consistent} with the true parameter $\bbeta_0$, the same as the assumption (A2) in \cite{yu2016sparse}. The Assumption \ref{assup2} (2) serves a similar role as the restricted eigenvalue condition for Lasso \citep{bickel2009simultaneous} to guarantee the oracle properties of the estimate. Intuitively, Assumption \ref{assup2} (2) requires that nonzero effects are strong enough to enable reliable estimation and to distinguish between nonzero and zero effects. 
Compared with the assumption (A3) for the data matrix in \cite{yu2016sparse}, here the assumption is for the Fisher information matrix due to a different loss function, $-l(\bbeta)$, considered here.


Under the assumptions above, we present the finite-sample recovery error for the maximum regularized likelihood estimate $\widehat\bbeta$, as summarized in Theorem~\ref{thm:oracle}. 

\begin{theorem}[Finite Sample Bounds]\label{thm:oracle}
Under the Assumptions \ref{assup1} and \ref{assup2}, let $\tau_{\min} = \min_{1\leq i\leq p}\tau_i$. For the optimal solution $\widehat\bbeta$ of problem \eqref{eqn:opt_graph}, there exist constants $D,D',K,K'$ such that with probability at least $1-pD e^{-Kn\lambda^2 \tau_{\min}^2/p} - p^2D'e^{-K'n \tau_{\min}^4\kappa^2/p^2}$, we have
\[
\Vert \widehat\bbeta - \bbeta_0\Vert_{2}  \leq \frac{12\lambda}{\kappa\tau_{\min}}.
\]
\end{theorem}
Theorem~\ref{thm:oracle} shows that the recovery error $\Vert \widehat\bbeta - \bbeta_0\Vert_{2}$ may be large when the smallest restricted eigenvalue $\kappa$ as imposed in Assumption \ref{assup2} (2) is close to zero, and the recovery error tends to be small when the regularizer parameter $\lambda$ is small.
\textcolor{black}{The result in Theorem~\ref{thm:oracle} is also consistent with the results for the linear regression model in \cite{yu2016sparse}. The main difference from \cite{yu2016sparse} is that the parameter $\kappa$ here is inherently determined by the Cox model itself, while in \cite{yu2016sparse}, a similar parameter appears from the restricted eigenvalue condition for the linear model.} 


We also derive the asymptotic normality property of the maximum regularized likelihood estimate under the case that the dimension $p$ of the covariate is fixed. We adopt the convention that $\bbeta_{J_0}$, $\bbeta_{J_0^c}$ denote the subvectors of $\bbeta$ consisting of entries with index belonging to the set $J_0$ and $J_0^c$, respectively, and $I_{J_0}(\bbeta_0)$ denotes the square matrix with rows and columns belong to the set $J_0$.
\begin{theorem}[Asymptotic Normality]\label{thm:asynormal}
When dimension $p$ is fixed, assume $\sqrt{n}\lambda \rightarrow 0$ and $\tau_j = O(1)$, $\forall j \in J_0$, $n^{(\gamma + 1)/2}\lambda \rightarrow \infty$, and $\lim\inf_{n\rightarrow\infty}n^{-\gamma/2}\tau_j > 0$ for each $j\in J_0^c$, under Assumptions \ref{assup2} (1), we have as $n\rightarrow \infty$,  
\[
\sqrt{n}(\widehat\bbeta_{J_0} - \bbeta_{0,J_0}) \overset{d}{\rightarrow} N(0, I_{J_0}(\bbeta_0)^{-1}), \quad \widehat\bbeta_{J_0^c} \overset{d}{\rightarrow} 0.
\]
\end{theorem}
The result in Theorem~\ref{thm:asynormal} indicates that the proposed estimate is asymptotically consistent when the dimension $p$ is fixed and $n\to\infty$, in the sense that the support of the true parameter can be recovered. It can also provide approximate confidence intervals for the estimate when the sample size is moderately large, based on the asymptotic normal distribution. It is worthwhile remarking that while the true Fisher information matrix $I(\bbeta_0)$ could be unknown in practice, we may estimate it using the empirical Fisher information and using the coefficients estimate $\hat\bbeta$. 



\section{Simulation Study}\label{sec:simulation}

To evaluate the performance of the graph regularizer for the Cox model, it is compared with some existing regularizers for the Cox model, including the classical lasso \citep{tibshirani1996regression,tibshirani1997lasso}, ridge regression \citep{hoerl1970ridge}, elastic net \citep{zou2005regularization, wu2012elastic}, SCAD \citep{fan2001variable, fan2002variable}, and adaptive lasso (Alasso) \citep{zou2006adaptive, zhang2007adaptive}, see Section~\ref{sec:regularizers} for details. 
The regularized survival models are evaluated on the following performance measures:
\begin{enumerate}[(i)]
\item $\ell_{2}$ and $\ell_1$ errors of the estimated coefficients: $\lVert \widehat{\bbeta} - \bbeta_{0} \rVert_{2}$ and $\lVert \widehat{\bbeta} - \bbeta_{0} \rVert_{1}$;
\item Harrell's concordance index (c-index) \citep{harrell1996multivariable}. The c-index is a commonly used metric for evaluating survival prediction models. It measures the ability of the model to correctly predict the ranking of the survival time given a pair of new observations and is equivalent to the Area Under Curve (AUC) \citep{huang2005using}. A c-index of 0.5 is equivalent to random guessing, and 1 is a perfect prediction. In recent survival applications, a c-index between 0.6 and 0.7 is often considered satisfactory \citep{laimighofer2016unbiased}.
\end{enumerate}

Three types of predictor graph topologies are tested in the simulation study: (1) the sparse graph, (2) the ring graph, and (3) the graph with communities. Figure~\ref{fig: graph-topologies} illustrates the corresponding graph typologies. 

\begin{figure}[!ht]
\begin{center}
\includegraphics[width=0.6\textwidth]{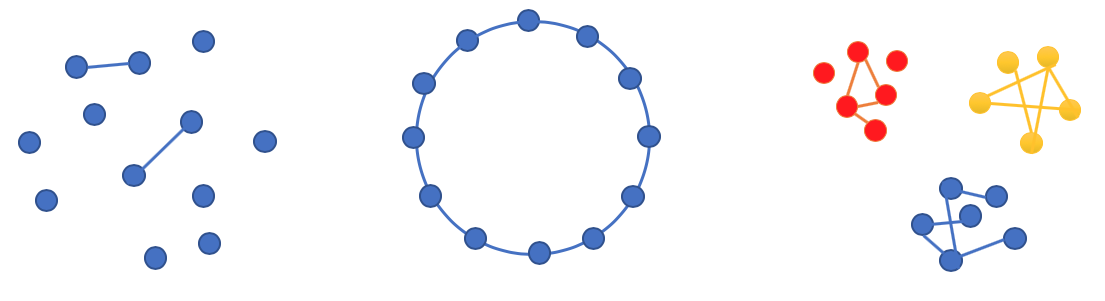}
\end{center}
\caption{Illustration of three predictor graph typologies used in the simulation. From left to right: the sparse graph, the ring graph, and the graph with three communities.}
\label{fig: graph-topologies}
\end{figure}

In the following, we show that the proposed graph regularizer has the most promising performance among the regularizers for the Cox model that was tested in the simulation study. \textcolor{black}{In the proposed estimation scheme, the tuning parameter is the regularization constant $\lambda$. In practice, this constant is chosen by cross-validation.
}



\subsection{Sparse Graph}

Consider a sparse Erd\H{o}s-R\'enyi predictor graph with a small edge formation probability $\rho_{0}$. Assume Gaussian distribution for predictors:  $(X_{1}, \ldots, X_{p})^{\top} \sim N(0, \Omega^{-1})$, where $\Omega$ is an inverse covariance matrix whose off-diagonal entries equal $0.5$ with probability $\rho_0$ and $0$ with probability $1-\rho_0$. In practice, we compute $\Lambda = \Omega^{-1}$ using the {\it nearPD} transformation in the \texttt{R} package \texttt{matrix} \citep{matrix} to ensure that $\Lambda$ is positive definite. As inspired by Remark \ref{rem1}, let the true parameters be $\bbeta_{0} = \Omega \Lambda_{xy}$, where $\Lambda_{xy} = (c_{1}, c_{2}, \dots, c_{p})^{\top}$. Let $c_{i} = 10$ for the top 4 predictors with maximum edges, and $c_{i} = 0$ otherwise. 
The survival time is simulated using the \texttt{R} package \texttt{coxed} \citep{kropko2019coxed} with a censor rate of 0.3. The training size is 100, and the testing size is 400. The hyper-parameters in each model are tuned by cross-validation using the training data.

The experiment is repeated 50 times, and the results (mean and standard deviation) of the models are shown in Table \ref{tab: sparse0.1p10} for a small covariate dimension $p=10$ and $\rho_0=0.1$; in Table \ref{tab: sparse0.01}, \ref{tab: sparse0.05}, and \ref{tab: sparse0.1} for a larger dimension $p=100$ and $\rho_{0} = 0.01, 0.05, 0.1$, respectively. We note that 50 iterations, based on empirical observations, are both computationally efficient and sufficiently demonstrate the relative performance of the methods with low variability. We see that under small dimension $p=10$, the proposed method based on graph regularize results in the best performance or very close to the best. In the setting of $p=100$, the graph structure tends to be more obvious and important for inference, and the proposed method based on graph regularizer results in lower $\ell_{2}$ and $\ell_1$ errors and higher c-index, compared to other regularizers and baseline models, regardless of the edge formation probability (see Table~\ref{tab: sparse0.01}, \ref{tab: sparse0.05}, \ref{tab: sparse0.1}). Here Inf value in the table means that the error magnitude is significantly larger than others (usually much larger than $10^4$).
%
As the edge formation probability $p_{0}$ increases, the performance of all models gets worse, but the graph-based regularization consistently results in better prediction than other models.  



\begin{table}[!ht]
\caption{Results on the Erd\H{o}s-R\'enyi predictor graph, $p=10$, $\rho_{0} = 0.1.$}
\centering
\begin{tabular}{ cccc } 
\specialrule{.08em}{0em}{0em}
 Model & $\ell_{2}$ norm & $\ell_{1}$ norm & c-index 
 \\ \hline
 \textbf{Graph regularizer} & 33.97 (0.49)  & 79.19 (0.95)  & {\bf 0.83} (0.03) \\
 Lasso& 34.26 (0.33)  & 79.67 (0.74)  & 0.83 (0.03) \\
 Ridge regression& 34.99 (0.1)  & 81.39 (0.29)  & 0.74 (0.05) \\
 Elastic net& 34.26 (0.32)  & 79.67 (0.72)  & 0.83 (0.03) \\
 SCAD& {\bf 33.74} (0.79)  & {\bf 78.61} (1.75)  & {\bf 0.83} (0.03) \\
 Alasso& 35.39 (0.01)  & 81.98 (0.02)  & 0.61 (0.03) \\
 Cox without regularization& 33.78 (0.42)  & 78.68 (0.95)  & 0.82 (0.03) \\
\specialrule{.08em}{0em}{0em}
\end{tabular}
\label{tab: sparse0.1p10}
\end{table}






 

\begin{table}[!ht]
\caption{Results on the Erd\H{o}s-R\'enyi predictor graph, $p=100$, $\rho_{0} = 0.01.$}
\centering
\begin{tabular}{ cccc } 
\specialrule{.08em}{0em}{0em}
 Model & $\ell_{2}$ norm & $\ell_{1}$ norm & c-index 
 \\ \hline
 \textbf{Graph regularizer} & {\bf 29.95  }(0.38) &{\bf 104.19 }(0.78) & {\bf 0.72} (0.04) 
 \\ 
Lasso & 30.30  (0.11) & 104.94 (0.19) &0.66 (0.04)\\ 
Ridge regression& 30.38 (0.03) & 105.45 (0.36) &0.60 (0.03)\\ 
Elastic net & 30.31 (0.08) & 105.01 (0.16) &0.66 (0.04)\\ 
SCAD & 30.34 (0.08) & 104.90 (0.18) & 0.66  (0.04) \\ 
Alasso & 30.40 (0.01) & 104.98 (0.02) &0.63  (0.05)\\ 
Cox without regularization& Inf(-) & Inf(-) & 0.54 (0.04)\\ 
\specialrule{.08em}{0em}{0em}
\end{tabular}
\label{tab: sparse0.01}
\end{table}

\begin{table}[!ht]
\caption{Results on the Erd\H{o}s-R\'enyi predictor graph, $p=100$, $\rho_{0} = 0.05.$}
\centering
\begin{tabular}{ cccc } 
\specialrule{.08em}{0em}{0em}
 Model & $\ell_{2}$ norm & $\ell_{1}$ norm & c-index 
 \\ \hline
 \textbf{Graph regularizer} & {\bf 41.95  }(0.48) &{\bf 238.25 }(1.66) & {\bf 0.70} (0.03) \\ 
Lasso & 42.27  (0.16) & 239.70 (0.37) &0.68 (0.03)\\ 
Ridge regression& 42.36 (0.06) & 240.25 (0.20) &0.66 (0.03)\\ 
Elastic net &  42.27  (0.15) & 239.78   (0.33) &0.67 (0.03)\\ 
SCAD & 42.37  (0.08) & 239.84 (0.26) & 0.68  (0.04) \\ 
Alasso & 42.41 (0.02) & 239.95 (0.05) &0.62  (0.06)\\  
Cox without regularization& Inf(-) & Inf(-) & 0.56 (0.05)\\ 
\specialrule{.08em}{0em}{0em}
\end{tabular}
\label{tab: sparse0.05}
\end{table}

 
 
 
 
 



\begin{table}[!ht]
\caption{Results on the Erd\H{o}s-R\'enyi predictor graph, $p=100$, $\rho_{0} = 0.1.$}
\centering
\begin{tabular}{ cccc } 
\specialrule{.08em}{0em}{0em}
 Model & $\ell_{2}$ norm & $\ell_{1}$ norm & c-index 
 \\ \hline
 \textbf{Graph regularizer} & {\bf 59.99} (0.55) &{\bf 372.66  }(1.85) & {\bf 0.70} (0.04) 
\\ 
Lasso & 60.46 (0.19) & 374.59 (0.59) &0.68 (0.03)\\ 
Ridge regression& 60.57 (0.05) & 375.12 (0.14) &0.66 (0.03)\\ 
Elastic net &  60.47  (0.15) & 374.66  (0.40) &0.68 (0.03)\\ 
SCAD & 60.57  (0.06) & 374.87 (0.19) & 0.67  (0.04) \\ 
Alasso & 60.60 (0.03) & 374.94   (0.08) &0.61  (0.06)\\  
Cox without regularization& Inf(-) & Inf(-) & 0.54 (0.05)\\ 
\specialrule{.08em}{0em}{0em}
\end{tabular}
\label{tab: sparse0.1}
\end{table}

Furthermore, we also conducted an ablation study for varying censoring mechanisms, including the covariate-independent censoring with varying censoring rates and the covariate-dependent censoring as shown in Appendix \ref{app}. It is observed that the results are robust to different censoring mechanisms, and the proposed method consistently outperforms other baseline methods. In addition to the estimation error and c-index, we also compare the number of non-zero coefficients selected by different penalty terms in order to demonstrate the interpretability of different methods. In practice, especially for problems with a large number of covariates, it is preferred to have a method that can select fewer variables while maintaining a similar level of accuracy since appropriate, but fewer variables typically imply better interpretability and variable selection. As shown in Table \ref{tab: sparse0.1-variable-num}, in this specific setting, the SCAD method tends to underestimate the number of non-zero coefficients, leading to an overly sparse solution, while graph regularizer selects a smaller number of coefficients as compared with all other methods and achieves a relatively high prediction accuracy.
\begin{table}[!ht]
\caption{Number of non-zero coefficients under the Erdős–Rényi predictor graph, where only coefficients with an absolute value greater than 0.1 are treated as non-zero to exclude negligible values. Averaged over 50 times.}
\centering
\begin{tabular}{ lccccccc } 
\specialrule{.08em}{0em}{0em}
\textbf{Model} & \textbf{Graph} & \textbf{Lasso} & \textbf{Ridge} & \textbf{Elastic net} & \textbf{SCAD} & \textbf{Alasso} & \textbf{Cox} \\ 
\hline
\textbf{$p=10$}  & 5.40  & 8.10  & 9.54  & 8.54  & 5.56  & 6.36  & 9.70 \\
\textbf{$p=100$} & 6.60  & 8.26  & 28.28 & 11.66 & 2.14  & 7.36  & 99.82 \\
\specialrule{.08em}{0em}{0em}
\end{tabular}
\label{tab: sparse0.1-variable-num}
\end{table}

\subsection{Ring Graph}

The second experiment we consider is on a ring predictor graph where the variables are nodes on the ring, and each node is connected to its immediate two neighbors, as shown in the middle of Figure\,\ref{fig: graph-topologies}. Let $(X_{1}, X_{2}, \dots, X_{p})^{\top} \sim N(0, \Omega^{-1})$, where $p = 100$. Let $\Omega = B + \delta I_p$, where $B_{ij} = 0.5$ for $|i-j| <2$ and $B_{ii} = 0$, $I_p$ is the identity matrix, and $\delta$ is chosen to make the condition number of $\Omega$ equal to $p$. Let the true parameter $\bbeta_{0} = \Omega \mathbf{1}$, where $\mathbf{1} \in \mathbb{R}^{p\times1}$ is a vector with all one entries. 

From the results in Table~\ref{tab: ring0p10} and Table~\ref{tab: ring0}, we observe that the graph-based regularizer has the best performance on the $\ell_{2}$ and $\ell_{1}$ errors, and the c-index when the predictor graph is a ring graph, for both $p=10$ and $p=100$ cases. 
The competing models have close performance with the graph regularizer since the relations among the variables in the ring graph are relatively simple. 

 
 
 
 
 



\begin{table}[!ht]
\caption{Performance on the ring predictor graph with $p=10$.}
\centering
\begin{tabular}{ cccc } 
\specialrule{.08em}{0em}{0em}
 Model & $\ell_{2}$ norm & $\ell_{1}$ norm & c-index 
 \\ \hline
 \textbf{Graph regularizer} & {\bf 40.99} (0.52)  & {\bf 92.63} (1.09)  & {\bf 0.85
 }(0.02)\\
 Lasso& 41.52 (0.33)  & 93.42 (0.77)  & 0.85 (0.02)\\
 Ridge regression& 41.52 (0.32)  & 93.52 (0.73)  & 0.84 (0.02)\\
 Elastic net& 41.53 (0.34)  & 93.5 (0.77)  & 0.85 (0.02)\\
 SCAD& 41.23 (0.83)  & 92.77 (1.86)  & 0.85 (0.02)\\
 Alasso& 42.6 (0.04)  & 95.67 (0.1)  & 0.84 (0.02)\\
 Cox without regularization& Inf(-)  & Inf(-)  & 0.83 (0.02)\\
\specialrule{.08em}{0em}{0em}
\end{tabular}
\label{tab: ring0p10}
\end{table}

\begin{table}[!ht]
\caption{Performance on the ring predictor graph with $p=100$.}
\centering
\begin{tabular}{ cccc } 
\specialrule{.08em}{0em}{0em}
 Model & $\ell_{2}$ norm & $\ell_{1}$ norm & c-index 
 \\ \hline
 \textbf{Graph regularizer} & {\bf 41.81} (0.36) &{\bf 94.85  }(0.69) & {\bf 0.79} (0.03) 
 \\       
 Lasso & 42.36 (0.27) & 95.74 (0.36) &0.74 (0.03)\\ 
Ridge regression& 42.73 (0.03) & 96.29 (0.38) &0.65 (0.03)\\ 
Elastic net &   42.49  (0.21) & 96.10  (0.28) &0.71 (0.03)\\ 
SCAD & 42.70   (0.06) & 95.97 (0.10) & 0.77  (0.04) \\ 
Alasso & 42.72 (0.02) & 95.98  (0.03) &0.69  (0.09)\\ 
Cox without regularization& Inf(-) & Inf(-) & 0.53 (0.04)\\ 
\specialrule{.08em}{0em}{0em}
\end{tabular}
\label{tab: ring0}
\end{table}



\subsection{Graph with Communities}

Suppose some of the predictors have community identities, and for predictors in the same community, an edge forms with probability $\rho_{\rm{in}}$; for predictors in different communities or those not in any communities, let the probability of edge formation among them be $\rho_{\rm{out}}$. Let $\rho_{\rm{in}} = 0.5, 0.7, 0.9$, and $\rho_{\rm{out}} = 0.01$. For covariate dimension $p = 100$ ($p=10$)
, we assume there exist three communities, each with size 30 (3), respectively.

The performance comparison is shown in Table~\ref{tab: commu0.5p10}, \ref{tab: commu0.9p10} for $p=10$, and 
in Table~\ref{tab: commu0.5}, \ref{tab: commu0.7}, \ref{tab: commu0.9} for $p=100$, under various $\rho_{\rm{in}}$ values, respectively. 
We observe that the graph-based regularization has the best $\ell_{2}$ norm and c-index regardless of the value of $\rho_{\rm{in}}$ in most cases, especially when the dimension is moderately large $p=100$. As $\rho_{\rm{in}}$ increases, the communities become denser, and the relations among the variables become more complex. Therefore, it becomes more difficult for the models to acquire accurate estimation and prediction. 

\begin{table}[!ht]
\caption{Results on the 3-community predictor graph, $p=10$, $\rho_{\rm{in}}  = 0.5.$}
\centering
\begin{tabular}{ cccc } 
\specialrule{.08em}{0em}{0em}
 Model & $\ell_{2}$ norm & $\ell_{1}$ norm & c-index 
 \\ \hline
 \textbf{Graph regularizer}  &  {\bf 12.24} (0.43)  & {\bf 31.37} (1.19)  & 0.86 (0.02) \\
 Lasso& 12.54 (0.21)  & 31.79 (0.72)  & 0.87 (0.02) \\
 Ridge regression& 12.51 (0.2)  & 31.69 (0.74)  & 0.87 (0.02) \\
 Elastic net& 12.59 (0.24)  & 31.83 (0.7)  & 0.87 (0.02) \\
 SCAD& 12.44 (0.31)  & 31.41 (1.06)  & 0.88 (0.02) \\
 Alasso& 13.11 (0.03)  & 33.72 (0.11)  & 0.87 (0.02) \\
 Cox without regularization& 14.14 (2.87)  & 35.35 (6.9)  & 0.85 (0.02) \\  
\specialrule{.08em}{0em}{0em}
\end{tabular}
\label{tab: commu0.5p10}
\end{table}


\begin{table}[!ht]
\caption{Results on the 3-community predictor graph, $p=10$, $\rho_{\rm{in}}  = 0.9.$}
\centering
\begin{tabular}{ cccc } 
\specialrule{.08em}{0em}{0em}
 Model & $\ell_{2}$ norm & $\ell_{1}$ norm & c-index 
 \\ \hline
 \textbf{Graph regularizer}  &  {\bf 19.87} (0.3)  & 46.77 (0.97)  & 0.86 (0.02) \\
  Lasso & 20.01 (0.26)  & {\bf 46.72} (0.69)  & 0.87 (0.02)\\  
 Ridge regression& 19.93 (0.29)  & 46.73 (0.67)  & 0.87 (0.02)\\  
 Elastic net& 20.29 (1.13)  & 48.07 (4.58)  & 0.87 (0.02)\\  
 SCAD& 20.66 (3.78)  & 48.63 (10.07)  & 0.87 (0.02)\\  
 Alasso & 20.77 (0.05)  & 48.61 (0.12)  & 0.86 (0.02)\\  
 Cox without regularization & 25.97 (4.2)  & 65.24 (12.11)  & 0.84 (0.03) \\  
\specialrule{.08em}{0em}{0em}
\end{tabular}
\label{tab: commu0.9p10}
\end{table}

\begin{table}[!t]
\caption{Results on the 3-community predictor graph, $p=100$, $\rho_{\rm{in}} = 0.5.$}
\centering
\begin{tabular}{ cccc } 
\specialrule{.08em}{0em}{0em}
 Model & $\ell_{2}$ norm & $\ell_{1}$ norm & c-index 
 \\ \hline
 \textbf{Graph regularizer} & {\bf 59.92} (0.78) &{\bf 432.26 }(3.01) & {\bf 0.69} (0.04) 
 \\                              
 Lasso & 60.57 (0.08) & 434.89 (0.19) &0.66 (0.03) \\ 
 Ridge regression& 60.60 (0.03) & 435.07 (0.11) & 0.64 (0.03) \\ 
 Elastic net &   60.58  (0.06) & 434.93  (0.12) &0.65 (0.04) \\ 
SCAD & 60.60   (0.03) & 434.96 (0.08) & 0.63  (0.04) \\ 
 Alasso & 60.62 (0.01) & 434.99  (0.02) &0.57  (0.05) \\ 
Cox without regularization& Inf(-) & Inf (-) & 0.54  (0.05)\\ 
\specialrule{.08em}{0em}{0em}
\end{tabular}
\label{tab: commu0.5}
\end{table}

\begin{table}[!t]
\caption{Results on the 3-community predictor graph, $p=100$, $\rho_{\rm{in}}  = 0.7.$}
\label{tab: commu0.7}
\centering
\begin{tabular}{ cccc } 
\specialrule{.08em}{0em}{0em}
 Model & $\ell_{2}$ norm & $\ell_{1}$ norm & c-index 
 \\ \hline
 \textbf{Graph regularizer} & {\bf 77.66} (0.49) &{\bf 525.16 }(2.98) & {\bf 0.69} (0.04) \\ 
 Lasso & 78.40 (0.05) & 529.97 (0.22) &0.65 (0.04) \\ 
 Ridge regression& 78.40 (0.03) & 530.09 (0.11) & 0.62 (0.04) \\ 
 Elastic net &   78.40  (0.04) & 529.96  (0.13) &0.64 (0.04) \\ 
SCAD & 78.41   (0.02) & 529.98 (0.10) & 0.61  (0.04) \\ 
 Alasso & 78.42 (0.01) & 530.00  (0.03) &0.55  (0.04) \\ 
Cox without regularization& Inf(-) & Inf (-) & 0.54  (0.05)\\ 
\specialrule{.08em}{0em}{0em}
\end{tabular}
\end{table}

 
 
 
 
 



\begin{table}[!ht]
\caption{Results on the 3-community predictor graph, $p=100$, $\rho_{\rm{in}}  = 0.9.$}
\centering
\begin{tabular}{ cccc } 
\specialrule{.08em}{0em}{0em}
 Model & $\ell_{2}$ norm & $\ell_{1}$ norm & c-index 
 \\ \hline
 \textbf{Graph regularizer} 
 & {\bf 89.80} (0.67) &{\bf 625.80 }(3.59) & {\bf 0.68} (0.03)  \\ 
 Lasso & 90.55 (0.01) & 630.01 (0.07) &0.57 (0.03) \\ 
 Ridge regression& 90.55 (0.01) & 630.08 (0.12) & 0.54 (0.02)\\ 
 Elastic net &   90.55  (0.01) & 630.02  (0.10) &0.56 (0.03) \\ 
SCAD & 90.55   (0.01) & 630.01 (0.09) & 0.55  (0.02) \\ 
 Alasso & 90.55 (0.00) & 630.00 (0.01) &0.53  (0.03) \\ 
Cox without regularization& Inf(-) & Inf (-) & 0.53  (0.04)\\ 
\specialrule{.08em}{0em}{0em}
\end{tabular}
\label{tab: commu0.9}
\end{table}

\section{Real Data Examples}\label{sec:data}

We apply the graph-based regularizer on two real datasets: the pediatric kidney transplant data and the primary biliary cirrhosis sequential (pbcseq) data, and compare performance with other commonly used regularization methods. 

\subsection{Pediatric Kidney Transplant Data}\label{sec:pediatric}

Predicting the survival time for transplant recipients is a crucial task for the transplant community. Accurate post-transplant survival prediction can provide helpful information for organ allocation decisions. A challenge with transplant survival prediction is that the data recorded for each transplant case are usually high-dimensional and highly correlated. Therefore, building a predictor graph and using the graph regularizer can be especially beneficial for solving the variable selection problems when building survival prediction models. 

We use the proposed graph regularized Cox model to predict the survival time of pediatric kidney transplant recipients. The dataset we use contains 19,236 pediatric kidney transplant cases in the U.S. from 1987 to 2014, and for each transplant case, there are 487 predictors. The dataset is provided by UNOS (United Network for Organ Sharing).  
The donor type, i.e., living versus deceased, can significantly impact the post-transplant survival \citep{gjertson2001determinants, terasaki1995high}; hence,  we separate the observations into two data sets and develop corresponding regularized Cox models for observations with living and deceased donors, respectively. 
We construct the predictor graph by connecting numerical predictors with high inverse covariance and connecting the paired categorical variables between the transplant recipient and the donor. For example, we connect the variables ``HBV: positive'' (recipient HBV infection status: positive) and ``HBV\_DON: positive'' (donor HBV infection status: positive). This connection is based on the assumption that being in a similar condition as the donor is beneficial for the organ recipient's post-transplant survivability. 


\begin{table}[!b]
\caption{Performance comparison on pediatric kidney transplant data.}
\centering
\begin{tabular}{ ccc } 
\specialrule{.08em}{0em}{0em}
 Model & Living donors c-index  & Deceased donors c-index\\ 
 \hline
 \textbf{Graph regularizer} &  {\bf 0.59}(0.045) &{\bf 0.58}(0.055)\\ 
 Lasso &0.57(0.039)&0.57(0.055)\\
 Ridge regression &0.49(0.039)&0.56(0.060)\\
 Elastic net &0.57(0.038)&0.58(0.045)\\ 
 SCAD &0.57(0.028)&0.57(0.056)\\
 Alasso & 0.57(0.040)&0.57(0.049)\\ 
 Group lasso & 0.57(0.051)&0.57(0.038)\\ 
 Cox without regularization &0.49(0.039)&0.55(0.058)\\
\specialrule{.08em}{0em}{0em}
\end{tabular}
\label{tab: transplantResult}
\end{table}

\begin{figure}[!ht]
\begin{center}
\begin{tabular}{c}
\includegraphics[width=0.6\textwidth]{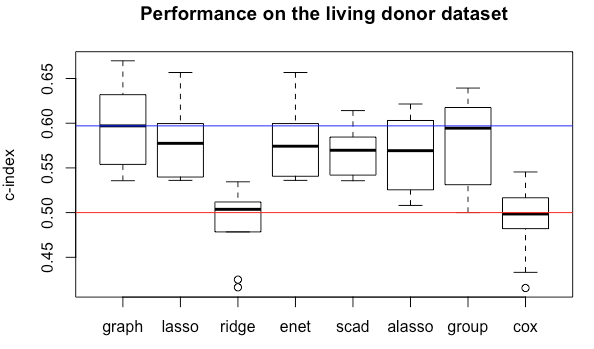} \\ \includegraphics[width=0.6\textwidth]{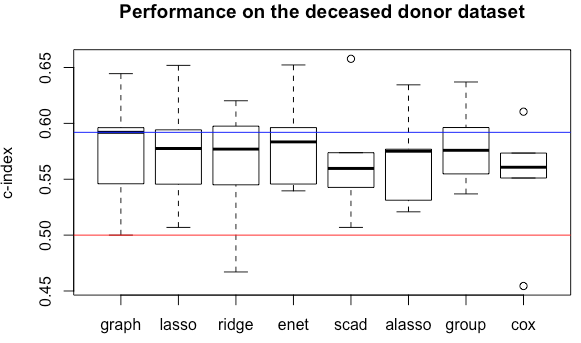}
\end{tabular}
\end{center}
\caption{The boxplot of the model c-indices on the living (upper) and deceased (bottom) donor datasets. The blue line indicates the median c-index of the proposed method; the red line indicates where the c-index equals 0.5 (random guessing).}
\label{fig: livingResult}
\end{figure}



The performance of the graph regularized Cox model is compared with other regularizers in Table \ref{tab: transplantResult} and Figures \ref{fig: livingResult}. Since the true parameters are unknown in the real data, we only compare the c-index. The results are based on five randomized partitions of the dataset. In each trial, 80\% of the data is used for training and 20\% for testing. The training and testing datasets are randomly partitioned in every trial to enhance robustness and reduce potential biases associated with a single data split. We observe that the graph-based regularizer has the highest mean and median c-index for both donor types. The improvement of using the graph regularizer is more prominent in the living donor dataset. This result is possible due to the fact that the living donor is more often related to the recipient and is likely to have closer biological and environmental characteristics than the recipient. More variables are also recorded from the living donors than from the deceased donors in the dataset. Therefore, the living donor predictor graph we can create is more complicated than the deceased donor graph, which gives the graph-based regularizer more advantages over other methods in predicting the survival outcome for pediatric recipients of living donor kidneys. 

We would like to emphasize that, as commented in Remark \ref{rem1}, the graph is treated as a fixed input parameter in our algorithm. When the ground truth graph structure is not perfectly known, we form the graph based on the correlation graph of observed data (for continuous variables) and domain knowledge or insights for all other variables (including the interaction of continuous and discrete variables). The estimation results may vary depending on the input graph structure. Therefore, in practice, when the graph structure is highly uncertain, we may also explore the fitting results across all potential graph structures and select the optimal one using cross-validation.

Moreover, taking the living donor dataset as an example, we present the estimated coefficients of variables identified as non-zero in Table \ref{tab:coefficients}. This provides a meaningful interpretation of the fitted model, and the small number of non-zero coefficients demonstrates the proposed method's ability to fit a model with few predictors. In the survival model, a positive coefficient indicates an increased hazard, meaning the variable is associated with a higher risk of mortality and thus has a negative impact on survival. Among the five selected variables, the most influential variable ``DIAG\_KI: TUBULAR.AND.INTERSTITIAL.DISEASES'' exhibits the largest positive coefficient, suggesting that a diagnosis of kidney tubular and interstitial diseases is strongly associated with an increased risk of mortality.

\begin{table}[ht]
\centering
\caption{The estimated coefficients of selected variables under the graph regularizer.}
\begin{tabular}{|c|l|r|}
\hline
\textbf{Variable Name} & \textbf{Coefficient} \\ \hline
DIAG\_KI: TUBULAR.AND.INTERSTITIAL.DISEASES & 0.4388 \\ \hline
ETHCAT: other & -0.0029 \\ \hline
EXH\_PERIT\_ACCESS: Y & 0.0021 \\ \hline
HAPLO\_TY\_MATCH\_DON: 1 & 0.0009 \\ \hline
HCV\_DON: unknown & 0.0338 \\ \hline
\end{tabular}
\label{tab:coefficients}
\end{table}

%

\subsection{Primary Biliary Cirrhosis Sequential ({\it pbcseq}) Data}
\label{sec:data_pbcseq}

The \textit{pbcseq} data \citep{murtaugh1994primary, fleming2011counting} in the \texttt{R} package \texttt{survival} \citep{survival} is recorded by the Mayo Clinic to study the primary biliary cirrhosis (PBC) of the liver from 1974 to 1984. It contains information on 1945 patients and 17 predicting variables. After removing the missing data, the pre-processed survival dataset contains 1113 samples in total.

To create a predictor graph, we analyze the relations of the variables in the {\it pbcseq} dataset. For the numerical variables, we compute their inverse covariance (shown in Figure~\ref{fig: pbcseq}). We connect pairs of variables if their Pearson's test $p$-value is less than 0.05 \citep{kim2015ppcor}. For the categorical variables, we connect variables representing different levels under the same categorical variable. For completeness, we summarize the variable relations for the predictor graph in Table \ref{tab: pbcseqNeighbors} in the Appendix. The neighbors of a variable are those that are connected to the variable. 

\begin{figure}[!t]
\begin{center}
\includegraphics[width=0.5\textwidth]{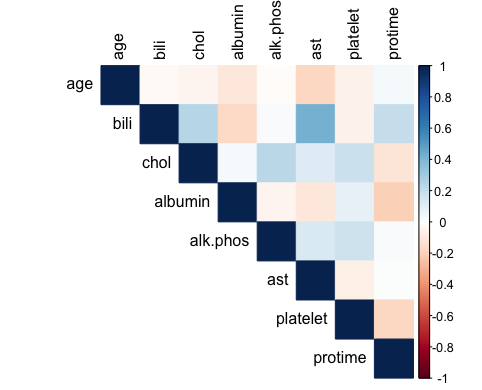}
\end{center}
\caption{Inverse covariance of the numerical variables in the {\it pbcseq} dataset.}
\label{fig: pbcseq}
\end{figure}


We compare the performance of the graph regularization to other methods using 10-fold cross-validation on the {\it pbcseq} dataset. Since this is a real data problem and the true parameters are unknown, only the c-index can be computed. We employed a similar approach to the previous kidney transplant dataset. The experimental results are based on 10 randomized partitions of the data. For each trial, 90\% of the data (1,002 samples) are used for training, and 10\% (111 samples) are used for testing/evaluation. Randomized partitioning is again performed in each trial to maintain robustness and mitigate biases from a single data split. The results are shown in Table \ref{tab: pbcseqResult} and Figure \ref{fig: pbcseqResult}, where the blue reference line in the figure is the median of the graph lasso c-index.

As shown in Table \ref{tab: pbcseqResult}, the graph-based regularizer has the highest c-index on the {\it pbcseq} dataset. The ridge regression, the elastic net, and the SCAD penalties also perform well. The boxplot shows that the graph-based regularization has the highest median c-index. The ridge regression and the elastic net have about the same median c-index as the graph regularizer. Still, their distributions of the c-index are lower than the graph lasso. 

Therefore, we can conclude that the graph-based regularization has satisfactory performance on the {\it pbcseq} dataset. However, its performance improvement is limited by the fact that the problem is not high-dimensional ($p = 17$), and the graphical structure among the variables is relatively simple. 

\begin{table}[!ht]
\caption{Performance of different penalties on {\it pbcseq} dataset.}
\centering
\begin{tabular}{ cc } 
\specialrule{.08em}{0em}{0em}
 Model & c-index \\ \hline
 
 \textbf{Graph regularizer} &  {\bf 0.88}(0.086)\\ 
 
 Lasso &0.86(0.082)\\ 
 
 Ridge regression &0.87(0.092)\\ 
 
 Elastic net &0.87(0.085)\\ 
 
 SCAD &0.87(0.079)\\ 
 
 Alasso & 0.86(0.088)\\ 
 
 Group lasso & 0.86(0.076)\\ 
 
 Cox without regularization &0.83(0.098)\\ 
\specialrule{.08em}{0em}{0em}

\end{tabular}
\label{tab: pbcseqResult}
\end{table}

\begin{figure}[!ht]
\begin{center}
\includegraphics[width=0.6\textwidth]{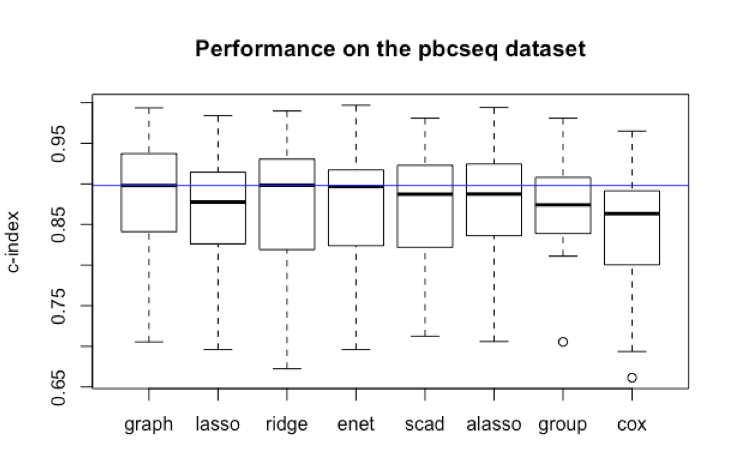}
\end{center}
\caption{The boxplot of the model c-indices on the {\it pbcseq} dataset. The blue line indicates the median c-index of the proposed method.}
\label{fig: pbcseqResult}
\end{figure}

\begin{table}[!ht]
\caption{Number of selected non-zero coefficients for pbcseq dataset, thresholding by 0.1.}
\centering
\begin{tabular}{ c|c|c|c|c|c|c|c} 
\specialrule{.08em}{0em}{0em}
\textbf{Graph} & Group Lasso & Lasso & Ridge &  Elastic net & SCAD & Alasso &
 Cox \\
 \hline
7 & 11 & 12 & 13 & 10 & 3 & 7 & 13\\
\specialrule{.08em}{0em}{0em}
\end{tabular}
\label{tab: pbcseq-variable-num}
\end{table}
We also compare the number of non-zero coefficients selected by different penalty terms as shown in Table \ref{tab: pbcseq-variable-num}. We note that in this real data example, the SCAD regularizer tends to underestimate the number of non-zero coefficients, leading to an overly sparse solution, while the graph regularizer selects a smaller number of coefficients as compared with all other methods and achieves a good prediction performance.

\section{Conclusion and Discussions}\label{sec:conclusion}

In this paper, we have studied the variable selection problem in survival analysis and developed a new graph-based regularized maximum partial-likelihood approach based on the Cox proportional hazard model. The graph-based regularization enables us to capture the complex graph-structured correlation between variables and, thus, more accurate variable selection compared to existing methods. We demonstrate the improved performance of our method compared with the state-of-the-art on simulated and real datasets. Although the problem is motivated by the organ transplantation application, the proposed method is very general and applicable to other applications where variable dependence can be captured through a graph.

There are several possible directions for future work. First, the choice of the regularization parameters $\lambda$ and $\bm\tau$ is critical to variable selection since there is typically a trade-off between sparsity and accuracy. Therefore, it would be useful to study the cross-validation for the Cox model under graph-based regularization, especially when we have censored data. The corresponding theoretical development is worth further investigation. Second, the graph structure used in this paper is only for predicting variables. We can also consider possible networks between donors and recipients for specific applications, such as the organ transplantation problem. Moreover, the graph can be further generalized to weighted graphs where weights may indicate the probability of success between each pair or the correlation of each pair of predicting variables.

\section*{Acknowledgments}
The authors are grateful to Professor Guan Yu for the helpful discussions.

\section*{Declarations}
{\bf Funding} The work of Xi He, Liyan Xie, and Yao Xie was partially supported by an NSF CAREER CCF-1650913, NSF DMS-2134037, CMMI-2015787, CMMI-2112533, DMS-1938106, DMS-1830210, and the Coca-Cola Foundation.

\noindent{\bf Competing interests} The authors have no conflict of interest to declare that are relevant to the content of this article.



\begin{appendices}

\section{Additional Implementation Details and Numerical Results}\label{app}


\noindent {\bf Additional Details on the Numerical Examples:}  Table\,\ref{tab: pbcseqNeighbors} gives a detailed description of the connected variable in the \textit{pbcseq} dataset.

\begin{table}[ht!]
\caption{Connected variables and their neighbors in the \textit{pbcseq} dataset used in Section~\ref{sec:data_pbcseq}.}
\centering
\begin{tabular}{ll } 
\specialrule{.08em}{0em}{0em}
 Variable & Neighbors \\  \hline
age& albumin, ast\\ 
bili& chol, albumin, ast, plateleet, protime\\ 
chol& alk.phos, ast, platelet, protime\\ 
albumin& ast, platelet, protime \\ 
alk.phos& ast, platelet\\ 
ast& platelet\\ 
platelet& protime\\ 
edema0.5& edema1  \\ 
stage2& stage3, stage4 \\
\specialrule{.08em}{0em}{0em}
\end{tabular}
\label{tab: pbcseqNeighbors}
\end{table}

\noindent {\bf Additional Discussions on the Solution Algorithm:} To derive the corresponding FISTA algorithm, we first study the quadratic approximation of $-\frac1n\ell(\bbeta)+\Vert \bbeta \Vert_{G,\bm\tau}$ at a given point $\bbeta'$,
\[
\begin{aligned}
Q_L(\bbeta,\bbeta')  =& -\frac1n\ell(\bbeta') + \langle \bbeta-\bbeta', -\frac1n\nabla \ell(\bbeta')  \rangle +\frac{L}{2n}\Vert \bbeta-\bbeta' \Vert_2^2 + \Vert \bbeta \Vert_{G,\bm\tau}.
\end{aligned}
\]
The quadratic approximation $Q_L(\bbeta,\bbeta')$ admits a unique minimizer
\[
p_L(\bbeta')= \arg\min_{\bbeta} \Bigg\{ \Vert \bbeta \Vert_{G,\bm\tau} + \frac{L}{2n}\left\Vert \bbeta - \bbeta' - \frac{\nabla \ell(\bbeta')}{L} \right \Vert^2 \Bigg\}.
\]
And the corresponding FISTA algorithm can be summarized in Algorithm \ref{alg:fista}.

\begin{algorithm}[ht!]
\caption{FISTA for solving $\bbeta$}\label{alg:fista}
\label{alg}
\SetAlgoLined
{\bf Input:} $L$ - the Lipschitz constant of $\nabla \ell(\bbeta)$\; 
{\bf Step 0:} Set $\bbeta_1=\bm z_1$ (pre-set initial value), $t_1 = 1$\;
\For {$k = 1,2,\cdots$}{
$\bm z_k = p_L(\bbeta_k)$;\\
$t_{k+1} = \frac{1+\sqrt{1+4t_k^2}}{2}$;\\
$\bbeta_{k+1} = \bm z_k + \frac{t_k-1}{t_{k+1}}(\bm z_k - \bm z_{k-1})$.\\
}
\end{algorithm}

\noindent {\bf Ablation Study: Varying Censoring Mechanisms.}

We conduct ablation studies under varying censoring rates and mechanisms to demonstrate the robust performance of the proposed method. For example, the results for censoring rate = 50\%, 40\%, 20\% for the sparse graph with $p=100$ are provided in Table \ref{tab: sparse0.1-censor05}, Table \ref{tab: sparse0.1-censor04}, Table \ref{tab: sparse0.1-censor02}, respectively. Furthermore, we also tried two scenarios where the censoring rate depends on the covariates:
\begin{itemize}
    \item Default setting in R. Table \ref{tab: sparse0.1-censor-dep} presents the results under the covariate-dependent censoring as the default configuration in R.  Specifically, by setting ``censor.cond'' to be TRUE within the \texttt{sim.survdata} function for simulating the survival data, then censoring depends on the covariates as follows: new coefficients are drawn from normal distributions with mean 0 and standard deviation of 0.1, and these new coefficients are used to create a new linear predictor using the $X$ matrix. The observations with the largest censoring percentage of the linear predictors are designated as right-censored.
    \item Censoring rate only depends on the top four covariates, with results presented in Table \ref{tab: sparse0.1-censor-dep1}. Specifically, we let the censoring rate equal to $\frac{1}{1+e^{0.85 + 0.1 (x_1 + x_2 + x_3 + x_4)}}$ where $x_1,\ldots,x_4$ are the top four covariates (corresponding to the four variables with the largest absolute value of regression coefficients). Here the values within the censoring rate function are chosen such that the average censoring rate across the entire simulated dataset is also controlled as 30\%.
\end{itemize}
Overall, the proposed method has robust performance to different censoring mechanisms and censoring rates.


\begin{table}[!ht]
\caption{Results on the Erd\H{o}s-R\'enyi predictor graph, $p=100$, $\rho_{0} = 0.1$, {censoring rate = 50\%}.}
\centering
\begin{tabular}{ cccc } 
\specialrule{.08em}{0em}{0em}
 Model & $\ell_{2}$ norm & $\ell_{1}$ norm & c-index 
 \\ \hline
 \textbf{Graph regularizer} & \textbf{60.34} (0.40) & \textbf{373.83} (1.50) & \textbf{0.65} (0.04) \\
 Lasso & 60.56 (0.08) & 374.93 (0.30) & 0.63 (0.04) \\
 Ridge regression & 60.59 (0.04) & 375.14 (0.18) & 0.60 (0.04) \\
 Elastic net & 60.56 (0.08) & 374.94 (0.28) & 0.62 (0.04) \\
 SCAD & 60.59 (0.05) & 374.96 (0.16) & 0.63 (0.04) \\
 Alasso & 60.47 (0.20) & 374.75 (0.76) & 0.60 (0.05) \\
 Cox without regularization & Inf (-) & Inf (-) & 0.56 (0.05) \\
\specialrule{.08em}{0em}{0em}
\end{tabular}
\label{tab: sparse0.1-censor05}
\end{table}
 
\begin{table}[!ht]
\caption{Results on the Erd\H{o}s-R\'enyi predictor graph, $p=100$, $\rho_{0} = 0.1$, {censoring rate = 40\%}.}
\centering
\begin{tabular}{ cccc } 
\specialrule{.08em}{0em}{0em}
 Model & $\ell_{2}$ norm & $\ell_{1}$ norm & c-index 
 \\ \hline
 \textbf{Graph regularizer} & \textbf{60.16} (0.53) & \textbf{373.24} (1.78) & \textbf{0.68} (0.04) \\
 Lasso & 60.52 (0.12) & 374.81 (0.33) & 0.66 (0.04) \\
 Ridge regression & 60.58 (0.05) & 375.15 (0.21) & 0.63 (0.04) \\
 Elastic net & 60.53 (0.10) & 374.85 (0.29) & 0.66 (0.04) \\
 SCAD & 60.58 (0.06) & 374.89 (0.21) & 0.65 (0.04) \\
 Alasso & 60.39 (0.23) & 374.51 (0.72) & 0.64 (0.05) \\
 Cox without regularization & Inf (-) & Inf (-) & 0.56 (0.05) \\
\specialrule{.08em}{0em}{0em}
\end{tabular}
\label{tab: sparse0.1-censor04}
\end{table}

\begin{table}[!ht]
\caption{Results on the Erd\H{o}s-R\'enyi predictor graph, $p=100$, $\rho_{0} = 0.1$, {censoring rate = 20\%}.}
\centering
\begin{tabular}{ cccc } 
\specialrule{.08em}{0em}{0em}
 Model & $\ell_{2}$ norm & $\ell_{1}$ norm & c-index 
 \\ \hline
 \textbf{Graph regularizer} & \textbf{59.82} (0.65) & \textbf{372.00} (2.14) & \textbf{0.74} (0.04) \\
 Lasso & 60.40 (0.18) & 374.40 (0.62) & 0.73 (0.04) \\
 Ridge regression & 60.57 (0.06) & 375.06 (0.09) & 0.70 (0.04) \\
 Elastic net & 60.42 (0.15) & 374.49 (0.52) & 0.72 (0.04) \\
 SCAD & 60.56 (0.07) & 374.85 (0.20) & 0.70 (0.04) \\
 Alasso & 60.18 (0.29) & 373.77 (1.06) & 0.70 (0.05) \\
 Cox without regularization & Inf (-) & Inf (-) & 0.56 (0.06) \\
\specialrule{.08em}{0em}{0em}
\end{tabular}
\label{tab: sparse0.1-censor02}
\end{table}

\begin{table}[!ht]
\caption{Results on the Erd\H{o}s-R\'enyi predictor graph, $p=100$, $\rho_{0} = 0.1$, {censoring rate = 30\%}, and training size = testing size = 100. Covariate dependent censoring (default in R).}
\centering
\begin{tabular}{ cccc } 
\specialrule{.08em}{0em}{0em}
 Model & $\ell_{2}$ norm & $\ell_{1}$ norm & c-index 
 \\ \hline
 \textbf{Graph regularizer} & \textbf{59.74} (0.57) & \textbf{371.51} (1.88) & \textbf{0.70} (0.06) \\
 Lasso & 60.32 (0.19) & 374.44 (0.55) & 0.68 (0.06) \\
 Ridge regression & 60.52 (0.02) & 375.31 (0.09) & 0.66 (0.06) \\
 Elastic net & 60.32 (0.17) & 374.54 (0.47) & 0.68 (0.06) \\
 SCAD & 60.54 (0.11) & 374.84 (0.29) & 0.67 (0.06) \\
 Alasso & 60.10 (0.30) & 373.87 (1.01) & 0.66 (0.07) \\
 Cox without regularization & Inf (-) & Inf (-) & 0.56 (0.07) \\
\specialrule{.08em}{0em}{0em}
\end{tabular}
\label{tab: sparse0.1-censor-dep}
\end{table}

\begin{table}[!ht]
\caption{Results on the Erd\H{o}s-R\'enyi predictor graph, $p=100$, $\rho_{0} = 0.1$, {censoring rate = 30\%}, and training size = testing size = 100. Covariate-dependent censoring (on the top 4 covariates).}
\centering
\begin{tabular}{ cccc } 
\specialrule{.08em}{0em}{0em}
 Model & $\ell_{2}$ norm & $\ell_{1}$ norm & c-index 
 \\ \hline
 \textbf{Graph regularizer} & {\bf 59.56} (0.81) & {\bf 370.75} (2.64) & {\bf 0.71} (0.05) \\
 Lasso & 60.22 (0.21) & 374.09 (0.72) & 0.69 (0.05) \\
 Ridge regression & 60.52 (0.03) & 375.29 (0.14) & 0.66 (0.05) \\
 Elastic net & 60.27 (0.17) & 374.33 (0.51) & 0.68 (0.05) \\
 SCAD & 60.51 (0.09) & 374.76 (0.30) & 0.67 (0.05) \\
 Alasso & 59.95 (0.31) & 373.33 (1.21) & 0.66 (0.06) \\
 Cox without regularization & Inf (-) & Inf (-) & 0.56 (0.05) \\
\specialrule{.08em}{0em}{0em}
\end{tabular}
\label{tab: sparse0.1-censor-dep1}
\end{table}

%
\newpage
\section{Proofs}\label{app:proof}
In this appendix, we provide the proof of the main theorems presented in the paper. We first define some empirical counterparts for the corresponding population quantities in \eqref{eq:def_s}:
\begin{eqnarray*}
S^{(0)} (\bbeta,t) &=& \frac1n\sum_{i=1}^n Y_i(t)\exp\{\bbeta^{\top}\bm x_i\},  \\
S^{(1)} (\bbeta,t) &=& \frac1n\sum_{i=1}^n Y_i(t)\bm x_i \exp\{\bbeta^{\top}\bm x_i\} ,  \\
S^{(2)} (\bbeta,t) &=& \frac1n\sum_{i=1}^n Y_i(t) \bm x_i\bm x_i^{\top}\exp\{\bbeta^{\top}\bm x_i\}. 
\end{eqnarray*}
Therefore, the partial likelihood score function can be written as
\begin{equation}
U(\bbeta) =  \frac{1}{n}\frac{\partial \ell}{\partial \bbeta} = \frac1n\sum_{i=1}^n \int_0^1 \Big\{ \bm x_i - \frac{S^{(1)} (\bbeta,t) }{S^{(0)} (\bbeta,t) }\Big\}dN_i(t).  
\end{equation}
Furthermore, the empirical Fisher information matrix can be calculated as
\begin{equation}\label{eq:emp_Fisher}
\Sigma(\bbeta) \!=\! -\frac{\partial U(\bbeta)}{\partial \bbeta} \!=\! \frac1n\sum_{i=1}^n\int_0^1\Big\{ \frac{S^{(2)} (\bbeta,t) }{S^{(0)} (\bbeta,t) } - \big( \frac{S^{(1)} (\bbeta,t) }{S^{(0)} (\bbeta,t) }  \big)\big( \frac{S^{(1)} (\bbeta,t) }{S^{(0)} (\bbeta,t) }  \big)^{\top}\Big\} dN_i(t).      
\end{equation}

We then list some lemmas which will be used in the following proofs. The following lemma establishes the concentration property of the score function $U(\bbeta_0)$ around 0.
\begin{lemma}[\textnormal{[\citenum{sun2014network}, Lemma A.2]}]\label{lem:lemma3}
Under Assumptions\,\ref{assup1} (1-2), there exists constants $C,D,K$ such that 
\[
\mathbb P[|U_j(\bbeta_0)|\geq C n^{-1/2}(1+x)] \leq D e^{-K(x^2 \wedge n)},
\]
for all $x>0$ and $j=1,\ldots,p$, where $U_j(\bbeta_0)$ is the $j$-th entry of the score function $U(\bbeta_0)$.
\end{lemma} 

The following Lemma establishes the concentration of the empirical information matrix in a neighborhood of $\bbeta_0$. 
\begin{lemma}[\textnormal{[\citenum{sun2014network}, Lemma A.3]}]\label{lem:lemma4}
Under Assumptions\,\ref{assup1} (1-2), there exists constants $C',D',K'$ such that for a neighborhood $\mathcal B\subset \mathbb{R}^{p}$ of $\bbeta_0$, we have
\[
\begin{aligned}
& \mathbb P\Big\{\sup_{\bbeta\in\mathcal B} |I_{i,j}(\bbeta) - \Sigma_{i,j}(\bbeta)| \geq C'\sqrt{\frac{p}{n}}(1+x)\Big\} 
\leq  D' e^{-K'(px^2 \wedge n)},
\end{aligned}
\]
where $I_{i,j}(\bbeta)$ and $\Sigma_{i,j}(\bbeta)$ are the $(i,j)$-th entry of the Fisher information matrix $I(\bbeta)$ and empirical information matrix $\Sigma(\bbeta)$ defined in \eqref{eq:emp_Fisher}, respectively.
\end{lemma} 
We refer to the supplementary material of \cite{sun2014network} for the detailed proof of Lemma \ref{lem:lemma3} and Lemma \ref{lem:lemma4}. As a consequence of the above two Lemmas, we now present a similar concentration result for the ``restricted eigenvalue'' of the Fisher information in the neighborhood $\mathcal B$.


\begin{lemma}\label{lem:lem6}
Under Assumptions\,\ref{assup1} (1-2) and Assumption\,\ref{assup2}, there exists constants $D'',K''>0$ such that for a neighborhood $\mathcal B\subset \mathbb{R}^{p}$ of $\bbeta_0$, we have with probability at least $1-p^2D''e^{-K''n \tau_{\min}^4\kappa^2/p^2}$, 
\[
\begin{aligned}
&\inf_{\bbeta\in\mathcal B} \inf_{ \substack{\bm\xi \in \mathbb R^p \backslash \{0\}\\|J|\leq s_0}} 
\inf_{\substack{(V^{(1)},V^{(2)},\ldots,V^{(p)}) \in \mathcal{U}(\bm\xi) \\ \sum_{k\notin J}\tau_k \Vert V^{(k)} \Vert_2 \leq 3\sum_{k\in J}\tau_k \Vert V^{(k)} \Vert_2 }}\frac12 \frac{  (\sum_{k=1}^p V^{(k)})^{\top} \Sigma(\bbeta) (\sum_{k=1}^p V^{(k)}) }{\sum_{k\in J} \tau_k^2 \Vert V^{(k)} \Vert_2^2} 
\geq \frac{\kappa}{2}.
\end{aligned}
\]
\end{lemma}

\begin{proof}
By Lemma\,\ref{lem:lemma4}, we have with probability $1-p^2D''e^{-K''n\epsilon^2}$, $\sup_{\bbeta\in\mathcal B} |I_{i,j}(\bbeta) - \Sigma_{i,j}(\bbeta)| \leq \epsilon$ for every entry $(i,j)$. Note that when $\sup_{\bbeta\in\mathcal B} |I_{i,j}(\bbeta) - \Sigma_{i,j}(\bbeta)| \leq \epsilon$, we have
\[
\begin{aligned}
& \Big|(\sum_{k=1}^p V^{(k)})^{\top} (\Sigma(\bbeta) - I(\bbeta)) (\sum_{k=1}^p V^{(k)})\Big| \\
\leq & \epsilon \Big\Vert \sum_{k=1}^p V^{(k)}\Big\Vert_1^2 \leq \epsilon p \Big\Vert\sum_{k=1}^p V^{(k)} \Big\Vert_2^2\\
\leq & \epsilon p \frac{ (\sum_{k} \tau_k \Vert V^{(k)} \Vert_2)^2}{\tau_{\min}^2} \\
=& \epsilon p \frac{ (\sum_{k\notin J} \tau_k \Vert V^{(k)} \Vert_2 + \sum_{k\in J} \tau_k \Vert V^{(k)} \Vert_2)^2}{\tau_{\min}^2} \\
\leq & 16\epsilon p \frac{ (\sum_{k\in J} \tau_k \Vert V^{(k)} \Vert_2)^2}{\tau_{\min}^2}, 
\end{aligned}
\]
where the last inequality is due to the imposed condition that $\sum_{k\notin J}\tau_k \Vert V^{(k)} \Vert_2 \leq 3\sum_{k\in J}\tau_k \Vert V^{(k)} \Vert_2$. Thus we have
\[
\frac{|(\sum_{k=1}^p V^{(k)})^{\top} (\Sigma(\bbeta) - I(\bbeta)) (\sum_{k=1}^p V^{(k)})| }{(\sum_{k\in J} \tau_k \Vert V^{(k)} \Vert_2)^2 } \leq \frac{16\epsilon p}{\tau_{\min}^2}.
\]
Setting $\epsilon = \frac{\tau_{\min}^2\kappa}{16p}$ yields $ \frac{16\epsilon p}{\tau_{\min}^2} \leq \kappa$, thus we complete the proof. This result shows that with high probability, the empirical information matrix $\Sigma$ shares almost the same properties with the population information matrix $I$. 
\end{proof}

We also present a useful Lemma from \cite{yu2016sparse} regarding the optimal decomposition for the graph-based regularization term.  
\begin{lemma}[\textnormal{[\citenum{yu2016sparse}, Lemma 2]}]\label{lem:lemma2}
For any predictor graph G and positive weights $\tau_i$, suppose $V^{(1)}$, $V^{(2)}$, $\ldots$, $V^{(p)}$ is an optimal decomposition of $\bbeta\in \mathbb R^p$, then for any $S \subseteq \{ 1,2,\ldots,p\}$, $\{V^{(j)}, j\in S\}$ is also an optimal decomposition of $\sum_{j\in S} V^{(j)}$.
\end{lemma} 

Using the above Lemmas, below we present the complete proof to the finite sample bound in Theorem~\ref{thm:oracle}.
\begin{proof}[Proof of Theorem~\ref{thm:oracle}]
Suppose $\widehat\bbeta$ is the optimal solution to the regularization problem \eqref{eq:main}, then for any $\bbeta \in \mathbb R^p$, we have
\[
-\frac1n \ell(\widehat\bbeta) + \lambda\Vert\widehat\bbeta\Vert_{G,\tau} \leq 
       -\frac1n \ell(\bbeta) + \lambda\Vert\bbeta\Vert_{G,\tau}.
\]
Let $\bbeta = \bbeta_0$, we have
\begin{equation}\label{eq:betahat_likelihood}
\frac1n \left\{ \ell(\bbeta_0) - \ell(\widehat\bbeta)\right\} \leq  \lambda \left( \Vert\bbeta_0\Vert_{G,\tau} -\Vert\widehat\bbeta\Vert_{G,\tau} \right).    
\end{equation}
Let $\{ S^{(1)}, \ldots, S^{(p)} \}\in \mathcal U(\bbeta_0)$ be an arbitrary optimal decomposition of $\bbeta_0$, and let $\{ T^{(1)}, \ldots, T^{(p)} \}\in \mathcal U(\widehat\bbeta -\bbeta_0)$ be an arbitrary optimal decomposition of $\widehat\bbeta -\bbeta_0$. We have $\widehat\bbeta -\bbeta_0 = \sum_{i=1}^p T^{(i)}$. 

We first analyze the right-hand-side of \eqref{eq:betahat_likelihood}, 
By Assumption \ref{assup2} (1), we can choose $S^{(j)} = 0, \forall j \in J_0^c$, thus $\bbeta_0 =  \sum_{j\in J_0} S^{(j)}$, and
\[
\begin{aligned}
\Vert \widehat\bbeta\Vert_{G,\tau} & = \Vert \widehat\bbeta - \bbeta_0 + \bbeta_0 \Vert_{G,\tau} \\
& = \Vert \sum_{j\in J_0} T^{(j)}  + \sum_{j\notin J_0} T^{(j)}  +\sum_{j\in J_0} S^{(j)} \Vert_{G,\tau}  \\
& \geq \Vert \sum_{j\notin J_0} T^{(j)}  +\sum_{j\in J_0} S^{(j)} \Vert_{G,\tau}  - \Vert \sum_{j\in J_0} T^{(j)} \Vert_{G,\tau} \\
& = \Vert \sum_{j\notin J_0} T^{(j)}\Vert_{G,\tau}  +  \Vert \sum_{j\in J_0} S^{(j)} \Vert_{G,\tau}  - \Vert \sum_{j\in J_0} T^{(j)} \Vert_{G,\tau}.
\end{aligned}
\]
Note that $ \Vert \sum_{j\in J_0} S^{(j)} \Vert_{G,\tau}=\Vert \bbeta_0\Vert_{G,\tau}$, thus
\[
\begin{aligned}
\Vert \bbeta_0\Vert_{G,\tau} - \Vert \widehat\bbeta\Vert_{G,\tau} & \leq \Vert \sum_{j\in J_0} T^{(j)} \Vert_{G,\tau} - \Vert \sum_{j\notin J_0} T^{(j)}\Vert_{G,\tau}.
\end{aligned}
\]

For the left-hand-side of \eqref{eq:betahat_likelihood} by expressing the log-likelihood function as a quadratic function in a neighborhood of the true parameter $\bbeta_0$, similar to the technique used in \cite{andersen1982cox,fan2002variable}, we have
\begin{equation}\label{eq:taylor}
\frac1n\left\{\ell(\bbeta_0) - \ell(\widehat\bbeta)\right\} 
= U(\bbeta_0)^\top(\bbeta_0-\widehat\bbeta) +\frac12(\widehat\bbeta - \bbeta_0)^{\top} \Sigma(\bar{\bbeta})(\widehat\bbeta- \bbeta_0),
\end{equation}
where $\bar{\bbeta}$ is a point in the line segment between $\bbeta_0$ and $\widehat\bbeta$. By Lemma\,\ref{lem:lemma3}, with  probability $1-pD e^{-Kn\epsilon^2}$ we have $\Vert U(\bbeta_0)\Vert_\infty\leq \epsilon$. Under the event that $\Vert U(\bbeta_0)\Vert_\infty\leq \epsilon$, combining with the right-hand-side, we have
\[
\begin{aligned}
& \frac12(\widehat\bbeta - \bbeta_0)^{\top}  \Sigma(\bar{\bbeta}) (\widehat\bbeta- \bbeta_0) \\
\leq & \lambda(\Vert \sum_{j\in J_0} T^{(j)} \Vert_{G,\tau} - \Vert \sum_{j\notin J_0} T^{(j)}\Vert_{G,\tau}) + U(\bbeta_0)^\top(\widehat\bbeta-\bbeta_0) \\
\overset{(i)}{\leq}& \lambda(\Vert \sum_{j\in J_0} T^{(j)} \Vert_{G,\tau} - \Vert \sum_{j\notin J_0} T^{(j)}\Vert_{G,\tau}) + \epsilon\Vert\widehat\bbeta-\bbeta_0\Vert_1 \\
\overset{(ii)}{\leq} & \lambda(\Vert \sum_{j\in J_0} T^{(j)} \Vert_{G,\tau} - \Vert \sum_{j\notin J_0} T^{(j)}\Vert_{G,\tau}) + \epsilon \sqrt{p} \left\Vert \sum_{i=1}^p T^{(i)}\right\Vert_2 \\
\leq & \lambda(\Vert \sum_{j\in J_0} T^{(j)} \Vert_{G,\tau} - \Vert \sum_{j\notin J_0} T^{(j)}\Vert_{G,\tau}) + \epsilon \sqrt{p} (\Vert \sum_{j\in J_0} T^{(j)} \Vert_{2} + \Vert \sum_{j\notin J_0} T^{(j)}\Vert_{2})\\
\overset{(iii)}{\leq} & (\lambda + \frac{\epsilon\sqrt{p}}{\tau_{\min}})\Vert \sum_{j\in J_0} T^{(j)} \Vert_{G,\tau} - (\lambda - \frac{\epsilon\sqrt{p}}{\tau_{\min}})\Vert \sum_{j\notin J_0} T^{(j)} \Vert_{G,\tau},
\end{aligned}
\]
where the inequality (i) is due to $\Vert U(\bbeta_0)\Vert_\infty\leq \epsilon$, the inequality (ii) is due to the Cauchy-Schwarz inequality, and the inequality (iii) is due to Lemma \ref{lem:lemma2}.
Select $\epsilon$ such that $\frac{\epsilon\sqrt{p}}{\tau_{\min}}\leq \frac{\lambda}{2}$, then we have
\[
\frac12(\widehat\bbeta - \bbeta_0)^{\top}  \Sigma(\bar{\bbeta}) (\widehat\bbeta- \bbeta_0) \leq \frac{3}{2}\lambda \Vert \sum_{j\in J_0} T^{(j)} \Vert_{G,\tau} - \frac{\lambda}{2} \Vert \sum_{j\notin J_0} T^{(j)} \Vert_{G,\tau}.
\]
Furthermore, notice that the particle log-likelihood function \eqref{eq:likelihood} is concave. Indeed, $\bbeta^\top \bm{x}_i$ is a linear function of $\bbeta$, and $\log\Big(\sum_{j: y_{j}\ge y_{i}}\exp(\boldsymbol{\beta}^{\top}\bm{x}_{j})\Big)$ is a convex function of $\bbeta$ since the summation of log-convex functions is also log-convex \cite{boyd2004convex}. Therefore, $\Sigma(\bar{\bbeta})$ is positive semidefinite for any $\bar{\bbeta}$, yielding $\frac12(\widehat\bbeta - \bbeta_0)^{\top}  \Sigma(\bar{\bbeta}) (\widehat\bbeta- \bbeta_0) \geq 0$, thus we have $\Vert \sum_{j\notin J_0} T^{(j)} \Vert_{G,\tau} \leq 3\Vert \sum_{j\in J_0} T^{(j)} \Vert_{G,\tau} $.

Then based on Assumption\,\ref{assup2} and Lemma\,\ref{lem:lem6}, we have
with probability at least $1-p^2D''e^{-K''n \tau_{\min}^4\kappa^2/p^2}$,
\begin{equation}\label{eq:require2}
\frac12(\widehat\bbeta - \bbeta_0)^{\top} \{ \Sigma(\bar{\bbeta})\} (\widehat\bbeta- \bbeta_0) \geq \frac\kappa 2 (\sum_{j\in J_0} \tau_j \Vert T^{(j)} \Vert_2)^2.
\end{equation}
On the other hand, condition on the event that \eqref{eq:require2} holds and $\frac{\epsilon\sqrt{p}}{\tau_{\min}}\leq \frac{\lambda}{2}$, note that we have 
\[
\begin{aligned}
&\frac\kappa 2 (\sum_{j\in J_0} \tau_j \Vert T^{(j)} \Vert_2)^2   \leq \frac32\lambda \sum_{j\in J_0} \tau_j \Vert T^{(j)} \Vert_2 \\
\Rightarrow & \Vert\sum_{j\in J_0}T^{(j)}\Vert_{G,\tau} = \sum_{j\in J_0} \tau_j \Vert T^{(j)} \Vert_2 \leq \frac{3\lambda}{\kappa}.
\end{aligned} 
\]
Furthermore, 
\begin{align*}
\Vert \widehat\bbeta - \bbeta_0\Vert_{2} 
& =     \Vert \sum_{j=1}^p T^{(j)} \Vert_2 \leq \frac{\sum_{j=1}^p \tau_j \Vert  T^{(j)} \Vert_2}{\tau_{\min}} \\
&= \frac{\sum_{j\in J_0} \tau_j \Vert T^{(j)}\Vert_2 +\sum_{j\notin J_0} \tau_j \Vert T^{(j)}\Vert_2 }{\tau_{\min}} \\
& \leq \frac{4\sum_{j\in J_0} \tau_j \Vert T^{(j)}\Vert_2 }{ \tau_{\min}}\leq \frac{12\lambda}{\kappa\tau_{\min}}.
\end{align*}
We notice that the above results are achieved by conditioning on the event $\Vert U(\bbeta_0)\Vert_\infty\leq \lambda\tau_{\min}/(2\sqrt{p})$ and the event in \eqref{eq:require2}, which hold simultaneously with probability $1-pD e^{-Kn\lambda^2 \tau_{\min}^2/p} - p^2D''e^{-K''n \tau_{\min}^4\kappa^2/p^2}$. Thus the proof is completed.
\end{proof}

Next, we give the proof of the asymptotic normality result in Theorem~\ref{thm:asynormal}.
\begin{proof}[Proof to Theorem~\ref{thm:asynormal}]
For each $\bu \in \mathbb R^p$, define 
$$Q_n(\bu) = -\ell(\bbeta_0 + n^{-1/2}\bu) + n\lambda\Vert \bbeta_0 + n^{-1/2}\bu \Vert_{G,\tau}.
$$
Since $\widehat\bbeta$ is the maximum penalized likelihood estimate, we define
\[
\hat \bu := \sqrt{n}(\widehat\bbeta- \bbeta_0) = \arg\min_{\bu \in \mathbb R^p}Q_n(\bu).
\]
We consider the asymptotic regime, and based on the local asymptotic quadratic property for the partial likelihood function as shown in \cite{andersen1982cox,fan2002variable}, we can write
\begin{align*}
Q_n(\bu) - Q_n(\bm 0) =& \ell(\bbeta_0) -  \ell(\bbeta_0 + n^{-1/2}\bu)   + n\lambda( \Vert \bbeta_0+ n^{-1/2}\bu \Vert_{G,\tau} - \Vert \bbeta_0  \Vert_{G,\tau} ) \\
 =& \frac12 \bu^{\top} I(\bbeta_0)\bu + o_P(1)+  n\lambda( \Vert \bbeta_0 + n^{-1/2}\bu \Vert_{G,\tau} - \Vert \bbeta_0  \Vert_{G,\tau} ). 
\end{align*}
For the second term, we have
\[
\begin{aligned}
& \Vert \bbeta_0 + n^{-1/2}\bu \Vert_{G,\tau} - \Vert \bbeta_0  \Vert_{G,\tau}  \\
=&  \Vert (\bbeta_0 + n^{-1/2}\bu)_{J_0} \Vert_{G,\tau} - \Vert \bbeta_0  \Vert_{G,\tau} + \Vert ( n^{-1/2}\bu)_{J_0^c} \Vert_{G,\tau} .
\end{aligned}
\]
Suppose $V^{(1)}, \ldots, V^{(p)}$ is an optimal decomposition of $\bu$, then by triangle inequality we have
\[
\begin{aligned}
& n\lambda( \Vert (\bbeta_0 + n^{-1/2}\bu)_{J_0} \Vert_{G,\tau} - \Vert \bbeta_0  \Vert_{G,\tau} )  \leq \sqrt{n}\lambda \Vert \bu_{J_0} \Vert_{G,\tau}=  \sqrt{n}\lambda \sum_{j\in J_0} \tau_j \Vert V^{(j)} \Vert_2.
\end{aligned}
\]
If $\sqrt{n}\lambda \rightarrow 0$ and $\tau_j = O(1)$ for each $j \in J_0$, then for each fixed $\bu$, we have
\begin{equation}\label{eq:u_j0}
n\lambda( \Vert (\bbeta_0 + n^{-1/2}\bu)_{J_0} \Vert_{G,\tau} - \Vert \bbeta_0  \Vert_{G,\tau} ) \rightarrow 0, \text{ as } n \rightarrow \infty.
\end{equation}
If $n^{(\gamma + 1)/2}\lambda \rightarrow \infty$, $\bu_{J_0^c} \neq 0$, and $\lim\inf_{n\rightarrow\infty}n^{-\gamma/2}\tau_j > 0$ for each $j\in J_0^c$, then 
\begin{equation}\label{eq:u_jc}
\begin{aligned}
n\lambda\Vert (n^{-1/2}\bu)_{J_0^c} \Vert_{G,\tau}  = &\sqrt{n}\lambda \sum_{j\in J_0^c} \tau_j \Vert V^{(j)} \Vert_2 \\
 = &n^{(\gamma + 1)/2}\lambda \cdot n^{-\gamma/2} \sum_{j\in J_0^c} \tau_j \Vert V^{(j)} \Vert_2 \rightarrow \infty.  
\end{aligned}
\end{equation}
Combining \eqref{eq:u_j0} and \eqref{eq:u_jc}, we have:
\[
\begin{aligned}
& Q_n(\bu) - Q_n(0)  \overset{d}{\rightarrow} \begin{cases}
\ell(\bbeta_0) -  \ell(\bbeta_0 + n^{-1/2}\bu) & \text{if supp}(\bu) \subset J_0, \\
\infty & \text{o.w.} 
\end{cases}
\end{aligned}
\]
This implies that 
\[
\widehat\bbeta_{J_0^c} \overset{d}{\rightarrow} 0.
\]
We note that $\hat\bu =\arg\min \{Q_n(\bu) - Q_n(0)\}$, thus it suffices to show that the $\hat\bu = \arg\max_{\text{supp}(\bu) \subset J_0} l(\bbeta_0 + n^{-1/2}\bu)$ is asymptotically normal distributed. 
To prove this, recall the first-order derivative of the partial log-likelihood with respect to $\bbeta$ is $U(\bbeta)$ and the second-order derivative is $-\Sigma(\bbeta)$. 
Using Taylor expansion, we have
\[
U(\widehat\bbeta) - U(\bbeta_0) = -\Sigma(\bbeta^*)(\widehat\bbeta - \bbeta_0),
\]
where $\bbeta^*$ is on the line segment between $\widehat\bbeta$ and $\bbeta_0$, and $\Sigma(\bbeta)$ is a positive semidefinite matrix. By Theorem 3.2 in \cite{andersen1982cox}, we have as $n \rightarrow \infty$,
\[
\frac{1}{\sqrt{n}} U_{J_0}(\bbeta_0) \overset{d}{\rightarrow} N(0,I_{J_0}(\bbeta_0)),  \ \frac{1}{n}\Sigma(\bbeta^*) \overset{p}{\rightarrow} I_{J_0}(\bbeta_0),
\]
where $U_{J_0}(\bbeta_0)$ consists of the elements of $U(\bbeta_0)$ with index belonging to the set $J_0$, and $I_{J_0}(\bbeta_0)$ is a square matrix with rows and columns belong to the index set $J_0$. 
Since $U(\widehat\bbeta) = 0$, we have $-\Sigma(\bbeta^*)(\widehat\bbeta - \bbeta_0) = U(\bbeta_0)$, thus by Slutsky's Theorem, we have
\[
\sqrt{n}I_{J_0}(\bbeta_0) (\widehat\bbeta_{J_0} - \bbeta_{0,J_0}) \overset{d}{\rightarrow} N(0,I_{J_0}(\bbeta_0)),
\]
which translates into
\[
\sqrt{n}(\widehat\bbeta_{J_0} - \bbeta_{0,J_0}) \overset{d}{\rightarrow} N(0, I_{J_0}(\bbeta_0)^{-1}).
\]
The proof is completed.
\end{proof}

\end{appendices}
\bibliography{graph_reg,references}

\begin{thebibliography}{79}
\providecommand{\natexlab}[1]{#1}
\providecommand{\url}[1]{{#1}}
\providecommand{\urlprefix}{URL }
\providecommand{\doi}[1]{\url{https://doi.org/#1}}
\providecommand{\eprint}[2][]{\url{#2}}
 \bibcommenthead

\bibitem[{Aalen(1978)}]{aalen1978nonparametric}
Aalen O (1978) Nonparametric inference for a family of counting processes. The
  Annals of Statistics pp 701--726

\bibitem[{de~Almeida~Costa et~al(2021)de~Almeida~Costa, de~Azevedo
  Peixoto~Braga, and Ramos~Andrade}]{de2021data}
de~Almeida~Costa M, de~Azevedo Peixoto~Braga JP, Ramos~Andrade A (2021) A
  data-driven maintenance policy for railway wheelset based on survival
  analysis and markov decision process. Quality and Reliability Engineering
  International 37(1):176--198

\bibitem[{Andersen and Gill(1982)}]{andersen1982cox}
Andersen PK, Gill RD (1982) Cox's regression model for counting processes: A
  large sample study. Annals of Statistics 10(4):1100--1120

\bibitem[{Athey et~al(2019)Athey, Tibshirani, and Wager}]{GRF-AoS}
Athey S, Tibshirani J, Wager S (2019) {Generalized random forests}. The Annals
  of Statistics 47(2):1148 -- 1178. \doi{10.1214/18-AOS1709},
  \urlprefix\url{https://doi.org/10.1214/18-AOS1709}

\bibitem[{Bates et~al(2017)Bates, Maechler, and Maechler}]{matrix}
Bates D, Maechler M, Maechler MM (2017) Package ‘matrix’

\bibitem[{Beck and Teboulle(2009)}]{beck2009fast}
Beck A, Teboulle M (2009) A fast iterative shrinkage-thresholding algorithm for
  linear inverse problems. SIAM journal on imaging sciences 2(1):183--202

\bibitem[{Bickel et~al(2009)Bickel, Ritov, and
  Tsybakov}]{bickel2009simultaneous}
Bickel PJ, Ritov Y, Tsybakov AB (2009) Simultaneous analysis of lasso and
  dantzig selector. Annals of Statistics 37(4):1705--1732

\bibitem[{Boyd and Vandenberghe(2004)}]{boyd2004convex}
Boyd SP, Vandenberghe L (2004) Convex optimization. Cambridge university press

\bibitem[{Breheny(2016)}]{breheny2016package}
Breheny P (2016) Package ‘grpreg’

\bibitem[{Chaturvedi et~al(2014)Chaturvedi, de~Menezes, and
  Goeman}]{chaturvedi2014fused}
Chaturvedi N, de~Menezes RX, Goeman JJ (2014) Fused lasso algorithm for {Cox}
  proportional hazards and binomial logit models with application to copy
  number profiles. Biometrical Journal 56(3):477--492

\bibitem[{Cox(1972)}]{coxregression}
Cox DR (1972) Regression models and life-tables. Journal of the Royal
  Statistical Society: Series B (Methodological) 34(2):187--202

\bibitem[{Cox(1975)}]{cox1975partial}
Cox DR (1975) Partial likelihood. Biometrika 62(2):269--276

\bibitem[{Duan et~al(2018)Duan, Zhang, Zhao, Shen, Wei, Chen, and
  Christiani}]{duan2018bayesian}
Duan W, Zhang R, Zhao Y, et~al (2018) Bayesian variable selection for
  parametric survival model with applications to cancer omics data. Human
  Genomics 12(1):1--15

\bibitem[{Fan and Li(2001)}]{fan2001variable}
Fan J, Li R (2001) Variable selection via nonconcave penalized likelihood and
  its oracle properties. Journal of the American statistical Association
  96(456):1348--1360

\bibitem[{Fan and Li(2002)}]{fan2002variable}
Fan J, Li R (2002) Variable selection for {Cox}'s proportional hazards model
  and frailty model. Annals of Statistics 30(1):74--99

\bibitem[{Fan et~al(2005)Fan, Li, and Li}]{fan2005overview}
Fan J, Li G, Li R (2005) An overview on variable selection for survival
  analysis. Contemporary Multivariate Analysis And Design Of Experiments: In
  Celebration of Professor Kai-Tai Fang's 65th Birthday pp 315--336

\bibitem[{Faraggi and Simon(1997)}]{faraggi1997large}
Faraggi D, Simon R (1997) Large sample {Bayesian} inference on the parameters
  of the proportional hazard models. Statistics in Medicine 16(22):2573--2585

\bibitem[{Faraggi and Simon(1998)}]{faraggi1998bayesian}
Faraggi D, Simon R (1998) Bayesian variable selection method for censored
  survival data. Biometrics 54(4):1475--1485

\bibitem[{Fleming and Harrington(2011)}]{fleming2011counting}
Fleming TR, Harrington DP (2011) Counting Processes and Survival Analysis, vol
  169. John Wiley \& Sons

\bibitem[{Giudici et~al(2003)Giudici, Mezzetti, and
  Muliere}]{giudici2003mixtures}
Giudici P, Mezzetti M, Muliere P (2003) Mixtures of products of {Dirichlet}
  processes for variable selection in survival analysis. Journal of Statistical
  Planning and Inference 111(1-2):101--115

\bibitem[{Gjertson and Cecka(2001)}]{gjertson2001determinants}
Gjertson DW, Cecka JM (2001) Determinants of long-term survival of pediatric
  kidney grafts reported to the united network for organ sharing kidney
  transplant registry. Pediatric Transplantation 5(1):5--15

\bibitem[{Greenland(1989)}]{greenland1989modeling}
Greenland S (1989) Modeling and variable selection in epidemiologic analysis.
  American Journal of Public Health 79(3):340--349

\bibitem[{Gross and Lai(1996{\natexlab{a}})}]{gross1996bootstrap}
Gross ST, Lai TL (1996{\natexlab{a}}) Bootstrap methods for truncated and
  censored data. Statistica Sinica pp 509--530

\bibitem[{Gross and Lai(1996{\natexlab{b}})}]{gross1996nonparametric}
Gross ST, Lai TL (1996{\natexlab{b}}) Nonparametric estimation and regression
  analysis with left-truncated and right-censored data. Journal of the American
  Statistical Association 91(435):1166--1180

\bibitem[{Gu and Lai(1998)}]{gu1998repeated}
Gu M, Lai TL (1998) Repeated significance testing with censored rank statistics
  in interim analysis of clinical trials. Statistica Sinica pp 411--428

\bibitem[{Gu and Lai(1990)}]{gu1990functional}
Gu MG, Lai TL (1990) Functional laws of the iterated logarithm for the
  product-limit estimator of a distribution function under random censorship or
  truncation. The Annals of Probability pp 160--189

\bibitem[{Gu and Lai(1991)}]{gu1991weak}
Gu MG, Lai TL (1991) Weak convergence of time-sequential censored rank
  statistics with applications to sequential testing in clinical trials. The
  Annals of Statistics pp 1403--1433

\bibitem[{Gu et~al(1991)Gu, Lai, and Lan}]{gu1991rank}
Gu MG, Lai TL, Lan KG (1991) Rank tests based on censored data and their
  sequential analogues. American Journal of Mathematical and Management
  Sciences 11(1-2):147--176

\bibitem[{Harrell~Jr et~al(1996)Harrell~Jr, Lee, and
  Mark}]{harrell1996multivariable}
Harrell~Jr FE, Lee KL, Mark DB (1996) Multivariable prognostic models: Issues
  in developing models, evaluating assumptions and adequacy, and measuring and
  reducing errors. Statistics in Medicine 15(4):361--387

\bibitem[{Hastie et~al(2015)Hastie, Tibshirani, and
  Wainwright}]{hastie2015statistical}
Hastie T, Tibshirani R, Wainwright M (2015) Statistical Learning with Sparsity:
  The Lasso and Generalizations. CRC press

\bibitem[{Hoerl and Kennard(1970)}]{hoerl1970ridge}
Hoerl AE, Kennard RW (1970) Ridge regression: Biased estimation for
  nonorthogonal problems. Technometrics 12(1):55--67

\bibitem[{Hothorn et~al(2004)Hothorn, Lausen, Benner, and
  Radespiel-Tr{\"o}ger}]{hothorn2004bagging}
Hothorn T, Lausen B, Benner A, et~al (2004) Bagging survival trees. Statistics
  in medicine 23(1):77--91

\bibitem[{Huang and Ling(2005)}]{huang2005using}
Huang J, Ling CX (2005) Using {AUC} and accuracy in evaluating learning
  algorithms. IEEE Transactions on Knowledge and Data Engineering
  17(3):299--310

\bibitem[{Ibrahim et~al(1999)Ibrahim, Chen, and
  MacEachern}]{ibrahim1999bayesian}
Ibrahim JG, Chen MH, MacEachern SN (1999) Bayesian variable selection for
  proportional hazards models. Canadian Journal of Statistics 27(4):701--717

\bibitem[{Jin and Lai(2017)}]{jin2017new}
Jin Y, Lai TL (2017) A new approach to regression analysis of censored
  competing-risks data. Lifetime data analysis 23:605--625

\bibitem[{Kaplan and Meier(1958)}]{kaplan1958nonparametric}
Kaplan EL, Meier P (1958) Nonparametric estimation from incomplete
  observations. Journal of the American statistical association
  53(282):457--481

\bibitem[{Khan and Shaw(2016)}]{khan2016variable}
Khan MHR, Shaw JEH (2016) Variable selection for survival data with a class of
  adaptive elastic net techniques. Statistics and Computing 26(3):725--741

\bibitem[{Kim and Lai(2000)}]{kim2000efficient}
Kim CK, Lai TL (2000) Efficient score estimation and adaptive m-estimators in
  censored and truncated regression models. Statistica Sinica pp 731--749

\bibitem[{Kim et~al(2012)Kim, Sohn, Jung, Kim, and Park}]{kim2012analysis}
Kim J, Sohn I, Jung SH, et~al (2012) Analysis of survival data with group
  lasso. Communications in Statistics-Simulation and Computation
  41(9):1593--1605

\bibitem[{Kim(2015)}]{kim2015ppcor}
Kim S (2015) ppcor: An {R} package for a fast calculation to semi-partial
  correlation coefficients. Communications for Statistical Applications and
  Methods 22(6):665

\bibitem[{Klein and Moeschberger(2003)}]{klein2003survival}
Klein JP, Moeschberger ML (2003) Survival analysis: techniques for censored and
  truncated data, vol 1230. Springer

\bibitem[{Kropko and Jeffrey(2019)}]{kropko2019coxed}
Kropko J, Jeffrey JH (2019) coxed: An {R} package for computing duration-based
  quantities from the {Cox} proportional hazards model.

\bibitem[{Lagani and Tsamardinos(2010)}]{lagani2010structure}
Lagani V, Tsamardinos I (2010) Structure-based variable selection for survival
  data. Bioinformatics 26(15):1887--1894

\bibitem[{Lai and Zheng(1993)}]{lai-book-survival}
Lai T, Zheng Z (1993) Survival Analysis (in Chinese). Zhejiang Publishing House
  of Science and Technology, Hangzhou

\bibitem[{Lai and Li(2006)}]{lai2006confidence}
Lai TL, Li W (2006) Confidence intervals in group sequential trials with random
  group sizes and applications to survival analysis. Biometrika 93(3):641--654

\bibitem[{Lai and Su(2006)}]{lai2006confidence2}
Lai TL, Su Z (2006) Confidence intervals for survival quantiles in the cox
  regression model. Lifetime Data Analysis 12:407--419

\bibitem[{Lai and Ying(1988)}]{lai1988stochastic}
Lai TL, Ying Z (1988) Stochastic integrals of empirical-type processes with
  applications to censored regression. Journal of multivariate analysis
  27(2):334--358

\bibitem[{Lai and Ying(1991{\natexlab{a}})}]{lai1991estimating}
Lai TL, Ying Z (1991{\natexlab{a}}) Estimating a distribution function with
  truncated and censored data. The Annals of Statistics pp 417--442

\bibitem[{Lai and Ying(1991{\natexlab{b}})}]{lai1991large}
Lai TL, Ying Z (1991{\natexlab{b}}) Large sample theory of a modified
  buckley-james estimator for regression analysis with censored data. The
  Annals of Statistics pp 1370--1402

\bibitem[{Lai and Ying(1991{\natexlab{c}})}]{lai1991rank}
Lai TL, Ying Z (1991{\natexlab{c}}) Rank regression methods for left-truncated
  and right-censored data. The Annals of Statistics pp 531--556

\bibitem[{Lai and Ying(1992{\natexlab{a}})}]{lai1992asymptotic}
Lai TL, Ying Z (1992{\natexlab{a}}) Asymptotic theory of a bias-corrected least
  squares estimator in truncated regression. Statistica Sinica pp 519--539

\bibitem[{Lai and Ying(1992{\natexlab{b}})}]{lai1992asymptotically}
Lai TL, Ying Z (1992{\natexlab{b}}) Asymptotically efficient estimation in
  censored and truncated regression models. Statistica Sinica pp 17--46

\bibitem[{Lai and Ying(1992{\natexlab{c}})}]{lai1992linear}
Lai TL, Ying Z (1992{\natexlab{c}}) Linear rank statistics in regression
  analysis with censored or truncated data. Journal of Multivariate analysis
  40(1):13--45

\bibitem[{Lai and Ying(1994)}]{lai1994missing}
Lai TL, Ying Z (1994) A missing information principle and m-estimators in
  regression analysis with censored and truncated data. The Annals of
  Statistics pp 1222--1255

\bibitem[{Lai et~al(1995)Lai, Ying, and Zheng}]{lai1995asymptotic}
Lai TL, Ying Z, Zheng Z (1995) Asymptotic normality of a class of adaptive
  statistics with applications to synthetic data methods for censored
  regression. Journal of Multivariate analysis 52(2):259--279

\bibitem[{Laimighofer et~al(2016)Laimighofer, Krumsiek, Buettner, and
  Theis}]{laimighofer2016unbiased}
Laimighofer M, Krumsiek J, Buettner F, et~al (2016) Unbiased prediction and
  feature selection in high-dimensional survival regression. Journal of
  Computational Biology 23(4):279--290

\bibitem[{Lee et~al(2011)Lee, Chakraborty, Sun et~al}]{lee2011bayesian}
Lee KH, Chakraborty S, Sun J, et~al (2011) Bayesian variable selection in
  semiparametric proportional hazards model for high dimensional survival data.
  The International Journal of Biostatistics 7(1):1--32

\bibitem[{Li et~al(2020)Li, Mark, Raskutti, Willett, Song, and
  Neiman}]{li2020graph}
Li Y, Mark B, Raskutti G, et~al (2020) Graph-based regularization for
  regression problems with alignment and highly correlated designs. SIAM
  Journal on Mathematics of Data Science 2(2):480--504

\bibitem[{Murphy and Van~der Vaart(2000)}]{murphy2000profile}
Murphy SA, Van~der Vaart AW (2000) On profile likelihood. Journal of the
  American Statistical Association 95(450):449--465

\bibitem[{Murtaugh et~al(1994)Murtaugh, Dickson, Van~Dam, Malinchoc, Grambsch,
  Langworthy, and Gips}]{murtaugh1994primary}
Murtaugh PA, Dickson ER, Van~Dam GM, et~al (1994) Primary biliary cirrhosis:
  Prediction of short-term survival based on repeated patient visits.
  Hepatology 20(1):126--134

\bibitem[{Nelson(1972)}]{nelson1972theory}
Nelson W (1972) Theory and applications of hazard plotting for censored failure
  data. Technometrics 14(4):945--966

\bibitem[{Newcombe et~al(2017)Newcombe, Raza~Ali, Blows, Provenzano, Pharoah,
  Caldas, and Richardson}]{newcombe2017weibull}
Newcombe PJ, Raza~Ali H, Blows FM, et~al (2017) Weibull regression with
  {Bayesian} variable selection to identify prognostic tumour markers of breast
  cancer survival. Statistical Methods in Medical Research 26(1):414--436

\bibitem[{Nikooienejad et~al(2020)Nikooienejad, Wang, and
  Johnson}]{nikooienejad2020bayesian}
Nikooienejad A, Wang W, Johnson VE (2020) Bayesian variable selection for
  survival data using inverse moment priors. Annals of Applied Statistics
  14(2):809

\bibitem[{Obozinski et~al(2011)Obozinski, Jacob, and Vert}]{obozinski2011group}
Obozinski G, Jacob L, Vert JP (2011) Group lasso with overlaps: The latent
  group lasso approach. arXiv preprint arXiv:11100413

\bibitem[{Ohno-Machado(2001)}]{ohno2001modeling}
Ohno-Machado L (2001) Modeling medical prognosis: survival analysis techniques.
  Journal of biomedical informatics 34(6):428--439

\bibitem[{Sun et~al(2014)Sun, Lin, Feng, and Li}]{sun2014network}
Sun H, Lin W, Feng R, et~al (2014) Network-regularized high-dimensional cox
  regression for analysis of genomic data. Statistica Sinica 24(3):1433

\bibitem[{Tachmazidou et~al(2010)Tachmazidou, Johnson, and
  De~Iorio}]{tachmazidou2010bayesian}
Tachmazidou I, Johnson MR, De~Iorio M (2010) Bayesian variable selection for
  survival regression in genetics. Genetic Epidemiology 34(7):689--701

\bibitem[{Terasaki et~al(1995)Terasaki, Cecka, Gjertson, and
  Takemoto}]{terasaki1995high}
Terasaki PI, Cecka JM, Gjertson DW, et~al (1995) High survival rates of kidney
  transplants from spousal and living unrelated donors. New England Journal of
  Medicine 333(6):333--336

\bibitem[{Therneau et~al(2020)Therneau, Lumley, Elizabeth, and
  Cynthia}]{survival}
Therneau TM, Lumley T, Elizabeth A, et~al (2020) Package `survival': Survival
  Analysis

\bibitem[{Tibshirani(1996)}]{tibshirani1996regression}
Tibshirani R (1996) Regression shrinkage and selection via the lasso. Journal
  of the Royal Statistical Society: Series B (Methodological) 58(1):267--288

\bibitem[{Tibshirani(1997)}]{tibshirani1997lasso}
Tibshirani R (1997) The lasso method for variable selection in the {Cox} model.
  Statistics in Medicine 16(4):385--395

\bibitem[{Utazirubanda et~al(2019)Utazirubanda, M.~Le{\'o}n, and
  Ngom}]{utazirubanda2019variable}
Utazirubanda JC, M.~Le{\'o}n T, Ngom P (2019) Variable selection with group
  lasso approach: Application to {Cox} regression with frailty model.
  Communications in Statistics-Simulation and Computation pp 1--21

\bibitem[{Walter and Tiemeier(2009)}]{walter2009variable}
Walter S, Tiemeier H (2009) Variable selection: Current practice in
  epidemiological studies. European Journal of Epidemiology 24(12):733--736

\bibitem[{Wu(2012)}]{wu2012elastic}
Wu Y (2012) Elastic net for {Cox}'s proportional hazards model with a solution
  path algorithm. Statistica Sinica 22:27--294

\bibitem[{Yu and Liu(2016)}]{yu2016sparse}
Yu G, Liu Y (2016) Sparse regression incorporating graphical structure among
  predictors. Journal of the American Statistical Association 111(514):707--720

\bibitem[{Yuan and Lin(2006)}]{yuan2006model}
Yuan M, Lin Y (2006) Model selection and estimation in regression with grouped
  variables. Journal of the Royal Statistical Society: Series B (Statistical
  Methodology) 68(1):49--67

\bibitem[{Zhang and Lu(2007)}]{zhang2007adaptive}
Zhang HH, Lu W (2007) Adaptive lasso for {Cox}'s proportional hazards model.
  Biometrika 94(3):691--703

\bibitem[{Zou(2006)}]{zou2006adaptive}
Zou H (2006) The adaptive lasso and its oracle properties. Journal of the
  American Statistical Association 101(476):1418--1429

\bibitem[{Zou and Hastie(2005)}]{zou2005regularization}
Zou H, Hastie T (2005) Regularization and variable selection via the elastic
  net. Journal of the Royal Statistical Society: series B (Statistical
  Methodology) 67(2):301--320

\end{thebibliography}

\end{document}